\documentclass[11pt]{llncs}
\usepackage[english]{babel}

\usepackage[letterpaper,top=2cm,bottom=2cm,left=3cm,right=3cm,marginparwidth=1.75cm]{geometry}

\usepackage{amsmath}
\usepackage{graphicx}
\usepackage[colorlinks=true, allcolors=blue]{hyperref}

\newcommand{\ccp}{{\it{ccp}}}
\newcommand{\p}{{\sf p}}   
\newcommand{\q}{{\sf q}}
\newcommand{\cP}{{\cal P}}
\newcommand{\cT}{{\cal T}}
\newcommand{\pb}{{\sf pb}}
\newcommand{\pw}{{\sf pw}}
\newcommand{\pl}{{\sf pl}}

\newcommand{\tl}{{\sf tl}}
\newcommand{\ipl}{{\sf ipl}}
\newcommand{\ipb}{{\sf ipb}}
\newcommand{\spb}{{\sf spb}}
\newcommand{\adc}{{\sf adc}}
\newcommand{\pat}{{\sf pat}}
\newcommand{\dpr}{{\sf dpr}}
\newcommand{\dsp}{{\sf dsp}}
\newcommand{\mci}{{\sf mci}}
\newcommand{\pcc}{{\sf pcc}}
\newcommand{\mfi}{{\sf mfi}}

\title{Graph parameters that are coarsely equivalent to path-length\thanks{\date*{\today}}}
\author{Feodor F. Dragan\inst{1} \and
Ekkehard K\"ohler\inst{2}  }

\institute{Computer Science Department, Kent State University, Kent, Ohio,  USA \\
\email{dragan@cs.kent.edu}  
\and 
Diskrete Mathematik und Grundlagen der Informatik, \\
Brandenburgische Technische Universit\"at Cottbus - Senftenberg, Cottbus, Germany\\
\email{koehlere@b-tu.de} 
}

\begin{document}
\maketitle

\begin{abstract}
Two graph parameters are said to be coarsely equivalent if they are within constant factors from each other for every graph $G$. 
Recently, several graph parameters were shown to be coarsely  equivalent to tree-length. Recall that the length of a tree-decomposition $\cT(G)$ of a graph $G$ is the largest diameter of a bag in  $\cT(G)$, and the  tree-length $\tl(G)$ of $G$ is the minimum of the length, over all tree-decompositions of $G$. Similarly, the length of a path-decomposition $\cP(G)$ of a graph $G$ is the largest diameter of a bag in  $\cP(G)$, and the  path-length $\pl(G)$ of $G$ is the minimum of the length, over all path-decompositions of $G$. In this paper, we present several graph parameters that are coarsely equivalent to path-length. Among other results, we show that the path-length of a graph $G$ is small if and only if one of the following equivalent conditions is true: 
\begin{itemize}
\item  $G$ can be embedded to an unweighted caterpillar tree (equivalently, to a graph of path-width one) with a small additive distortion; 
\item  there is a constant $r\ge 0$ such that for every triple of vertices $u,v,w$ of $G$, disk of radius $r$ centered at one of them intercepts all paths connecting the two others;
\item $G$ has a $k$-dominating shortest path with small $k\ge 0$; 
\item $G$ has a $k'$-dominating pair with small $k'\ge 0$; 
\item some power $G^\mu$ of $G$ is an AT-free (or even a cocomparability) graph for a small integer $\mu\ge 0$.  
\end{itemize} 
\end{abstract}

\section{Introduction} 
Two graph parameters $\p$ and $\q$ are said \cite{coarse-tl} to be {\em coarsely equivalent} if there are two universal constants $\alpha>0$ and $\beta>0$ such that $\alpha\cdot \q(G)\le \p(G)\le \beta\cdot \q(G)$ for every graph $G$.  So, if one parameter is bounded by a constant, then the other is bounded by a constant, too. 
Coarse equivalency of two graph parameters is useful in at least two scenarios \cite{coarse-tl}. First, if one parameter is easier to compute, then it provides an easily computable constant-factor approximate for possibly hard to compute other parameter. This is the case for parameters the  path-length $\pl(G)$ of a graph $G$ and the length $\Delta(G)$ of a best path-decomposition of $G$ obtained from its Breadth-First-Search layering (formal definitions of these and other parameters can be found in Section \ref{sec:pl-pb} and Section \ref{sec:new-par}). It is known \cite{DrKoLe2017}  that those two parameters are within small constant factors from each other. 
The length $\Delta(G)$ of a best path-decomposition of $G$ obtained from its Breadth-First-Search layering can easily be computed. This  provides an efficient 
2-approximation algorithm for computing the path-length of a graph (and a corresponding path-decomposition), which is NP-hard to compute exactly (see Section \ref{sec:pl-pb}  for details). 
Secondly, since a constant bound on one parameter implies a constant bound on the other, one can choose out of two a most suitable (a right one) parameter when designing FPT (approximation) algorithms for some particular optimization problems on bounded parameter graphs. For example, an $O(k)$-factor approximation algorithm for the so-called {\em minimum line distortion problem} on graphs admitting a $k$-dominating shortest path developed in \cite{DrLe2017}  provides also an $O(\pl(G))$-factor approximation algorithm for the problem on graphs with path-length $\pl(G)$ \cite{DrKoLe2017}. A similar implication holds also for the so-called {\em minimum bandwidth  problem} on graphs:  an $O(k)$-factor approximation algorithm for the problem on graphs admitting a $k$-dominating shortest path provides an $O(\pl(G))$-factor approximation algorithm for the problem on graphs with path-length $\pl(G)$ \cite{DrKoLe2017}. These two parameters (the path-length of $G$ and the minimum $k$ such that a $k$-dominating shortest path exists in $G$) are coarsely equivalent (see Section \ref{sec:pat}). All these motivate to look for graph parameters that are coarsely equivalent.   

This paper is also inspired by recent insightful papers  \cite{Diestel++,BerSey2024,coarse-tl,GeorPapa2023} where several graph parameters that are coarsely equivalent to tree-length were described. It is shown that the tree-length of a graph $G$ is bounded if and only if there is an  $(L,C)$-quasi-isometry (equivalently, an  $(1,C')$-quasi-isometry) to a tree with $L,C$ ($C'$, respectively) bounded (see Section \ref{sec:adc} for a definition).  
One of the main results of \cite{BerSey2024} was a proof of Rose McCarty's  conjecture that the tree-length of a graph is small if and only if its  McCarty-width is small (see Section \ref{sec:mci} for a definition and \cite{coarse-tl} for  simpler proofs). Among other results, it was shown in \cite{coarse-tl}  that the tree-length of a graph $G$ is bounded if and only if for every bramble ${\cal F}$ (or every Helly family of connected subgraphs ${\cal F}$, or every Helly family of paths ${\cal F}$)  of $G$,  there is a disk in $G$ with bounded radius that intercepts all members of ${\cal F}$. In \cite{coarse-tl} it was also shown that a generalization of a known  characteristic cycle property of chordal graphs (graphs with tree-length equal to 1) coarsely defines the tree-length. An interesting (coarse) characterization of tree-length was provided in  \cite{GeorPapa2023} (see also \cite{coarse-tl}) 
via $K$-fat $K_3$-minors: 
the tree-length of a graph $G$ is bounded by a constant if and only if $G$ has no $K$-fat $K_3$-minor for some constant $K>0$ (see Section \ref{sec:fat-minors} for a definition of a $K$-fat $K_3$-minor). 

In this paper, we introduce several graph parameters  (formal definitions can be found in Section \ref{sec:pl-pb} and Section \ref{sec:new-par}) and show that they all are coarsely equivalent to path-length. Among other results, we prove that the path-length of a graph $G$ is small if and only if one of the following equivalent conditions is true: 
\begin{enumerate}
\item[1.]  $G$ can be embedded to an unweighted caterpillar tree (equivalently, to a graph of path-width one) with a small additive distortion; 
\item[2.]  there is an  $(L,C)$-quasi-isometry (equivalently, an $(1,C')$-quasi-isometry) to a path with small $L\ge 1,C\ge 0$ (small $C'\ge 0$, respectively);
\item[3.]  the McCarty-index of $G$ (a linearized version of the  McCarty-width) is small;
\item[4.] $G$ has a $k$-dominating shortest path with small $k\ge 0$; 
\item[5.] $G$ has a $k'$-dominating pair with small $k'\ge 0$; 
\item[6.] some power $G^\mu$ of $G$ is an AT-free graph for a small integer $\mu\ge 0$;  
\item[7.] some power $G^{\mu'}$ of $G$ is a cocomparability graph for a small integer $\mu'\ge 0$: 
\item[8.] $G$ has neither $K$-fat $K_3$-minor nor $K$-fat $K_{1,3}$-minor for some constant $K>0$.
\end{enumerate}
It was known before \cite{DrKoLe2017} that if the path-length of a graph $G$ is $\lambda$ then $G$ has a $\lambda$-dominating pair (and, hence, a $\lambda$-dominating shortest path) and the ($2\lambda-1$)-power $G^{2\lambda-1}$ of $G$ is an AT-free graph.  It is an interesting question to ask to what extend the other directions hold, that is, are there constants $c$, $c'$ and  $c''$ such that if $G$ has a $k$-dominating pair, a $k'$-dominating shortest path, or $G^\mu$ is AT-free, then the path-length of $G$ is at most $ck$, at most $c'k'$, at most $c''\mu$, respectively? Here,  we answer that question in the affirmative. 
From recent results of \cite{Diestel++,GeorPapa2023}, it follows that the path-length of a graph $G$ is bounded by a constant if and only if $G$ has neither $K$-fat $K_3$-minor nor $K$-fat $K_{1,3}$-minor for some constant $K>0$.  Using the relations obtained in this paper between path-length and new coarsely equivalent parameters, we give also an alternative simple proof, with precise constants, of that result.

\medskip
\noindent
{\bf Basic notions and notations.} %
All graphs occurring in this paper are connected, finite, unweighted, undirected, loopless and without multiple edges. 
For a graph $G=(V,E)$, we use $n$ and $|V|$ interchangeably to denote the number of vertices in $G$. Also, we use $m$ and $|E|$ to denote the number of edges. When we talk about two or more graphs, we may use $V(G)$ and $E(G)$ to indicate that these are the vertex and edge sets of graph $G$.  A {\em clique} is a set of pairwise adjacent vertices of $G.$ By $G[S]$ we denote the subgraph of $G$ induced by the vertices of $S \subseteq V$. By $G \setminus S$ we denote the
subgraph of $G$ induced by the vertices $V \setminus S$, i.e., the graph $G[V \setminus S]$. For a vertex $v$ of $G$, the sets
$N_G(v) = \{w \in V : vw \in E\}$ and $N_G[v] = N_G(v) \cup \{v\}$ are called the {\em open neighborhood} and the
{\em closed neighborhood} of $v$, respectively.   

The {\em length of a path} $P(v,u):=(v=v_0,v_1,\dots,v_{\ell-1},v_{\ell}=u)$ from a vertex $v$ to a vertex $u$ is $\ell$, i.e., the number of edges in the path. The {\em distance} $d_G(u,v)$ between vertices $u$ and $v$ is the length of a shortest path connecting $u$ and $v$ in $G$. 
The distance between a vertex $v$ and a subset $S\subseteq V$ is defined as $d_G(v,S):=\min\{d_G(v,u): u\in S\}$. Similarly, let $d_G(S_1.S_2):=\min\{d_G(x,y): x\in S_1, y\in S_2\}$ for any two sets $S_1,S_2\subseteq V$. 
The $k^{th}$ {\em power} $G^k$ of a graph $G$ is a graph that has the same set of vertices, but in which two distinct vertices are adjacent if and only if their distance in $G$ is at most $k$. 
The {\em disk} $D_G(s,r)$ of a graph $G$ centered at vertex $s \in V$ and with radius $r$ is the set of all vertices with distance no more than $r$ from $s$ (i.e., $D_G(s,r)=\{v\in V: d_G(v,s) \leq r \}$). We may omit the graph name $G$ and write  $D(s,r)$ if the context is about only one graph. 

The \emph{diameter} of a subset $S\subseteq V$ of vertices of a graph $G$ is the largest distance  in $G$ between a pair of vertices of $S$, i.e., $\max_{u,v \in S}d_G(u,v)$. The \emph{inner diameter} of $S$  is the largest distance in $G[S]$ between a pair of vertices of $S$, i.e., $\max_{u,v \in S}d_{G[S]}(u,v)$.  The \emph{radius} of a subset $S\subseteq V$ of vertices of a graph $G$ is the minimum $r$ such that a vertex $v\in V$ exists with $S\subseteq D_G(v,r)$. The \emph{ inner radius} of a subset $S$ is the minimum $r$ such that a vertex $v\in S$ exists with $S\subseteq D_{G[S]}(v,r)$. When $S=V$, we get the {\em diameter} $diam(G)$ and the {\em radius} $rad(G)$ of the entire graph $G$.  


Definitions of graph parameters considered in this paper, 
as well as notions and notation local to a section, are given in appropriate sections. We realize that there are too many parameters and abbreviations and this may cause some difficulties in following them.  We give in Appendix a glossary for all parameters and summarize all inequalities and relations between them.  


\section{Path-length, path-breadth, and their variants.}\label{sec:pl-pb} 

An {\em interval graph} is the intersection graph of a set of intervals on the real line. A {\em caterpillar} or {\em caterpillar tree} is a tree in which all the vertices are within distance 1 of a central path. It is clear that every caterpillar is an interval graph. In fact, caterpillars are exactly the triangle-free interval graphs \cite{Eck1993}. 

The main object of this paper is Robertson-Seymour's path-decomposition and its length. A {\em path-decomposition} \cite{RobSey1983}  of a graph $G = (V, E)$ is a sequence of subsets $\{X_i : i \in I\}$ $(I := \{1, 2,\dots, q\})$ of
vertices of $G$, called {\em bags}, with three properties:
\begin{enumerate}
\item[(1)] $\bigcup_{i\in I}X_i = V$,

\item[(2)] For each edge $vw\in E$, there is a bag $X_i$
such that $v,w\in X_i$, and

\item[(3)] For every three indices $i\le j \le k,$ $X_i \cap X_k \subseteq X_j$. Equivalently, the subsets containing any
particular vertex form a contiguous subsequence of the whole sequence. 
\end{enumerate}

We denote a path-decomposition $\{X_i : i \in I\}$ of a graph $G$ by $\cP(G)$. The {\em width} of a path-decomposition $\cP(G) = \{X_i : i\in I\}$ is $\max_{i\in I} |X_i| - 1$. The {\em path-width} \cite{RobSey1983} of
a graph $G$, denoted by $\pw(G)$, is the minimum width over all path-decompositions $\cP(G)$ of $G$.
It is known that the caterpillars are exactly the graphs with path-width 1 \cite{PrTe1999}.   The {\em length}
of a path-decomposition $\cP(G)$ of a graph $G$ is $\lambda := \max_{i\in I} \max_{u,v\in X_i} d_G(u, v)$ (i.e., each bag $X_i$
has diameter at most $\lambda$ in $G$). The {\em path-length} \cite{DrKoLe2017} of $G$, denoted by $\pl(G)$, is the minimum length over
all path-decompositions of $G$. Interval graphs   
are exactly the graphs with path-length 1; it is known (see, e.g., \cite{FuGr1965,GilHof1964,GolBook}) that $G$ is an interval
graph if and only if $G$ has a path-decomposition with each bag being a maximal clique of $G$ (a so-called clique-path).
Note that these two graph parameters (path-width and path-length) are not related to each
other. For instance, a clique on $n$ vertices has path-length 1 and path-width $n - 1$, whereas a cycle
on $2n$ vertices has path-width 2 and path-length $n$.  
The {\em breadth} of a path-decomposition $\cP(G)$
of a graph $G$ is the minimum integer $r$ such that for every $i \in I$ there is a vertex $v_i \in V$ with
$X_i \subseteq D_G(v_i, r)$ (i.e., each bag $X_i$ can be covered by a disk $D_G(v_i, r)$ of radius at most $r$ in $G$).
Note that vertex $v_i$ does not need to belong to $X_i$. The {\em path-breadth} \cite{DrKoLe2017} of $G$, denoted by $\pb(G)$, is the
minimum breadth over all path-decompositions of $G$. 

Evidently, for any graph $G$ with at least one edge, $1 \le \pb(G) \le \pl(G) \le 2\cdot \pb(G)$ holds. Hence, if one parameter is bounded by a constant for a
graph $G$ then the other parameter is bounded for $G$ as well. 
We say that a family of graphs ${\cal G}$ is {\em of bounded path-length} (equivalently, {\em of bounded path-breadth}), if there is a constant $c$ such that for each graph  $G$ from ${\cal G}$,  $\pl(G)\leq c$. 

It is known  \cite{DuLeNi2016} that checking whether a graph $G$ satisfies  
$\pl(G)\le \lambda$ or $\pb(G)\le r$ is NP-complete for each $\lambda >1$ and each $r>0$. However,  a path decomposition of length at most $2\cdot\pl(G)$ and of breadth at most $3\cdot\pb(G)$ of an $n$ vertex  graph $G$ can be computed in $O(n^3)$ time~\cite{DrKoLe2017}. 
See also \cite{ArneDisser} for other interesting results on path-length and path-breadth. 

In \cite{Dulei2019,LeiDra2016}, a notion of strong breadth was also introduced. The {\em  strong breadth}  of a path-decomposition $\cP(G)$ of a graph $G$ is the minimum integer $r$ such that for every $i\in I$ there is a vertex $v_i\in X_i$ with $X_i= D_G(v_i,r)$ (i.e., each bag $X_i$ is equal to a disk  of $G$ of radius at most $r$). The {\em strong path-breadth}  of $G$, denoted by $\spb(G)$, is the minimum of the strong breadth, over all path-decompositions of $G$. Like for the path-breadth, it is NP-complete to determine if a given graph has strong path-breadth $r$, even for $r=1$ \cite{Du2018}. Clearly, $\pb(G)\le \spb(G)$ for every graph $G$. Furthermore, as it was shown in \cite{Dulei2019}, $\spb(G)\le 4\cdot\pb(G)$ for every graph $G$. Hence, path-breadth and strong path-breadth are two coarsely equivalent parameters.   


Notice that in the definition of the length of a path-decomposition, the distance between vertices of a bag is measured in the entire graph $G$. If the distance between any vertices $x,y$ from a bag $X_i$ $(i\in I)$ is measured in $G[X_i]$, then one gets the notion of the inner length of a path-decomposition (see \cite{BerSey2024}; it is called there {\em inner diameter-width}). The {\em inner length} of a path-decomposition $\cP(G)$ of a graph $G$ is $\max_{i\in I}\max_{u,v\in X_i}d_{G[X_i]}(u,v)$, and the {\em inner path-length} of $G$, denoted by $\ipl(G)$, is the minimum of the inner length, over all path-decompositions of $G$.  Similarly, the {\em  inner breadth}  of a path-decomposition $\cP(G)$ of a graph $G$ is the minimum integer $r$ such that for every $i\in I$ there is a vertex $v_i\in X_i$ with $d_{G[X_i]}(u,v_i)\le r$ for all $u\in X_i$ (i.e., each subgraph $G[X_i]$ has radius at most $r$). The {\em inner path-breadth}  of $G$, denoted by $\ipb(G)$, is the minimum of the inner breadth, over all path-decompositions of $G$ (see \cite{Diestel++};  it is called there {\em radial path-width}).  Interestingly, the inner path-length (inner path-breadth) of $G$ is at most twice the path-length (path-breadth, respectively) of $G$. That is,
	$\pl(G) \leq \ipl(G)\le 2\cdot\pl(G)$ and $\pb(G) \leq \ipb(G)\le 2\cdot\pb(G)$   for every graph $G$ \cite{Diestel++,BerSey2024}. These make all graph parameters $\pl(G),\pb(G),\spb(G),\ipl(G),\ipb(G)$ coarsely equivalent to each other. 
    
Since $\ipl(G)$ and $\ipb(G)$ are not that far from $\pl(G)$ and $\pb(G)$, in what follows, we will work only with $\pl(G)$ and $\pb(G)$.

\section{New parameters that are coarsely equivalent to path-length} \label{sec:new-par}
We introduce several more graph parameters and show that they are all coarsely equivalent to path-length. 

\subsection{Distance $k$-approximating caterpillars and quasi-isometry to paths}\label{sec:adc}
Recall that caterpillars are exactly the graphs with path-width 1. 
So, we coarsely characterize here the graphs that can be embedded to a graph of path-width 1 with a small (additive) distortion. 

A (unweighted) tree $T=(V,E')$ is a {\em distance $k$-approximating tree} of a graph $G=(V,E)$ \cite{DBLP:journals/jal/BrandstadtCD99,DBLP:journals/ejc/ChepoiD00} 
 if $|d_G(u,v)-d_T(u,v)|\le k$ holds for every $u,v\in V$. Denote by $\adc(G)$ the minimum $k$ such that $G$ has a distance $k$-approximating caterpillar tree.  We call $\adc(G)$ the {\em additive distortation} of embedding of $G$ to caterpillar tree $T$ ($T$ is with the same vertex set as $G$). 

\begin{lemma} \label{lm:adc}
	Let $k\ge 0$ be an integer, $G=(V,E)$ be a graph, and $T=(V,E')$ be a caterpillar tree (on the same vertex set) such that for every pair $x,y\in V$, $d_T(x,y)-k\le d_G(x,y)\le d_T(x,y)+k$ holds. Then, $\pl(G)\leq 2k+1$. 
\end{lemma}
\begin{proof} Let $G=(V,E)$ be a graph and $T=(V,E')$ be a caterpillar tree such that $|d_G(x,y)-d_T(x,y)|\le k$ for every $x,y\in V$. For every edge $uv$ of $G$, $d_T(u,v)\le d_G(u,v)+k\le 1+k$ holds. Hence, $G$ is a spanning subgraph of graph $T^{k+1}$, where  $T^{k+1}$ is the $(k+1)^{st}$-power of $T$. It is known (see, e.g., \cite{Andreas-book,Ray1987}) that every power of a caterpillar tree (in fact, of any interval graph) is an interval graph. Consequently, there is a clique-path $\cP$ of $T^{k+1}$.  Clearly, $\cP$ is a path-decomposition of $G$ such that, for every two vertices $x$ and $y$ belonging to same bag 
of $\cP$, $xy$ is an edge of $T^{k+1}$. Necessarily, $d_T(x,y)\le k+1$, by the definition of the $(k+1)^{st}$-power of  $T$. Furthermore, since $d_G(x,y)\le d_T(x,y)+k\le 2k+1$ for every $x,y$ belonging to same bag of $\cP$, the length of the path-decomposition $\cP$ of $G$ is at most $2k+1$. 
\qed
\end{proof}

Let $G = (V, E)$ be an arbitrary graph and let $s$ be an arbitrary vertex of $G$. A {\em layering} $L(s, G)$ of
$G$ with respect to a start vertex $s$ is the decomposition of $V$ into layers $L_i = \{u \in V : d_G(s, u) = i\}$, 
$i = 0, 1, \dots , q$ (where $q=\max\{d_G(u,s): u\in V\}$). Let $length(L(s, G)):=\max_{i=1}^q\max_{u,v\in L_i} d_G(u, v)$ and denote by $breadth(L(s, G))$ the minimum $r$ such that for every $i\in\{1,\dots,q\}$ there is a vertex $v_i$ in $G$ with $d_G(u, v_i)\le r$ for every $u\in L_i$. 
That is, each layer $L_i$ of layering $L(s, G)$ 
has diameter at most $length(L(s, G))$ and radius at most $breadth(L(s, G))$ in $G$. 
Let also $P=(x_0,x_1,\dots,x_q)$ be a (shortest) path of $G$ with $x_i\in L_i$. We have $s=x_0$. Given a layering $L(s, G)$ of $G$ and a path $P$, we can build a caterpillar tree $H_s=(V,E')$ for $G$, called a {\em canonical caterpillar} of $G$, as follows: for each $i= 1, \dots , q$ and each vertex $u\in L_i$, 
add an edge $ux_{i-1}$ to initially empty $E'$.  It is easy to show that this canonical caterpillar nicely approximates the distances in $G$. 

\begin{lemma} \label{lm:layer-length}
	Let $G=(V,E)$ be a graph and $L(s, G)$ be a layering of
$G$ with respect to any start vertex $s$. Then, every canonical caterpillar $H_s$ of $G$ satisfies the following inequalities for every $x,y\in V:$ $$d_{H_s}(x,y)-2\le d_G(x,y)\le d_{H_s}(x,y)+length(L(s, G))\le d_{H_s}(x,y)+2\cdot breadth(L(s, G)).$$
In particular, $\adc(G)\le length(L(s, G))$ for every graph $G$ and every vertex $s$ of $G$. 
\end{lemma}
\begin{proof} Let $\gamma:= length(L(s, G))$ and consider arbitrary two vertices $x,y$ of $G$. Assume $x\in L_i$ and $y\in L_j$ with $i\le j$.  By the definition of layering and the construction of $H_s$, we have $j-i\le d_G(x,y)\le j-i+ \gamma$ and $j-i\le d_{H_s}(x,y)\le j-i+ 2$. Consequently, $d_{H_s}(x,y)-2\le d_G(x,y)\le d_{H_s}(x,y)+\gamma.$ Note also that $length(L(s, G))\le 2\cdot breadth(L(s, G))$ as the diameter of a set is at most twice its radius.   
\qed
\end{proof}

We can also upper-bound $length(L(s, G))$ for some $s$ by a linear function on $\adc(G)$. 
\begin{lemma} \label{lm:layer-adc}
For every graph $G$,
there is a vertex $s$ such that 
$length(L(s, G))\le 3\cdot\adc(G)+2.$
\end{lemma}
\begin{proof}   Let $G=(V,E)$ be a graph and $T=(V,E')$ be a caterpillar tree such that $|d_G(x,y)-d_T(x,y)|\le k$ for every $x,y\in V$. Let $P$ be the central path of $T$ with end vertices $s$ and $t$. We show that $length(L(s, G))\le 3k+2$. Consider any two vertices $x$ and $y$ of $G$ with $d_G(s,x)=d_G(s,y)=\ell$, i.e., $x,y\in L_\ell$ in $L(s, G)$. Let $x', y'$ be the vertices of $P$ such that $x\in N_T[x']$ and $y\in N_T[y']$. Let also $d_T(s,x')=i$, $d_T(s,y')=j$ and assume $i\le j$.  

We have $j+d_T(y,y')=d_T(s,y)\le d_G(s,y)+k=\ell+k$, implying $j\le \ell+k-d_T(y,y')$. Consequently, $j-i\le \ell+k-d_T(y,y')-d_T(s,x')=\ell+k-d_T(y,y')-d_T(s,x)+d_T(x,x')$. 
Since $\ell= d_G(s,x)\le d_T(s,x)+k$, we obtain $j-i\le \ell+k-d_T(y,y')-\ell+k+d_T(x,x')$, i.e., $j-i\le 2k+d_T(x,x')-d_T(y,y')$. 

Finally, $d_G(x,y)\le d_T(x,y)+k= d_T(x,x')+(j-i)+d_T(y,y')+k\le d_T(x,x')+(2k+d_T(x,x')-d_T(y,y'))+d_T(y,y')+k= 3k+2d_T(x.x')\le 3k+2$. 
\qed
\end{proof}

As a byproduct of Lemma \ref{lm:layer-length} and Lemma \ref{lm:layer-adc}, we get an efficient ($O(n^3)$ time) approximation algorithm for the problem of minimum additive distortion embedding of a graph to a caterpillar tree. One just needs to construct all $n$ canonical caterpillar trees $\{H_s: s\in V\}$ of $G$ (in $O(m)$ time per caterpillar) and pick in $O(n^3)$ time that one which gives the smallest additive distortion. 

\begin{corollary}  \label{cor:emb-to-caterp-appr}
For every graph $G$, there is a vertex $s$ such that 
$\adc(G)\le length(L(s, G))\le 3\cdot\adc(G)+2.$ Furthermore, there is an efficient $(O(n^3)$ time$)$ algorithm that embeds any graph $G$ to a caterpillar tree with an additive distortion at most  $3\cdot \adc(G)+2.$ 
\end{corollary}

This approximation result might be of interest with respect to  the minimum bandwidth problem. It is known that the minimum bandwidth problem can be solved for caterpillar trees optimally in $O(n\log n)$ time~\cite{APSZ1981}, although it is an NP-hard problem on graphs with bounded path-width~\cite{DFU2011,Monien1986}  or with bounded path-length~\cite{DrKoLe2017}. 


For an integer $i \ge 1$ and a vertex $v \in 
L_i$, denote by $N_G^\downarrow (v) = N_G(v) \cap L_{i-1}$ the
neighborhood of $v$ in the previous layer $L_{i-1}$. We can get a path-decomposition of $G$ by adding
to each layer $L_i$ $(i > 0)$ all vertices from layer $L_{i-1}$ that have a neighbor in $L_i$, in particular, let
$L^+_i := L_i \cup (\bigcup_{v\in L_i} N^\downarrow_G(v))$. Clearly, the sequence $\{L^+_1 , \dots , L^+_q \}$ is a path-decomposition of $G$ and
can be constructed in $O(m)$ total time. We call this path-decomposition an {\em extended layering} of
$G$ and denote it by $L^+(s, G)$.  Let $\Delta_s(G)$ be the length and $\rho_s(G)$ be the breadth of the path-decomposition $L^+(s, G)$ of $G$.  In \cite{DrKoLe2017}, the following important result was proven. 

\begin{proposition} [\cite{DrKoLe2017}] \label{prop:ld-our}
For every graph $G$, 
there is a vertex $s$ such that the following holds for $L(s, G)$ and $L^+(s, G)$: $$length(L(s, G))\le \Delta_s(G)\le 2\cdot\pl(G) \mbox{~~ and~~} breadth(L(s, G))\le\rho_s(G)\le 3\cdot\pb(G).$$ 
\end{proposition}

Let $\Delta(G):=\min_{s\in V}\Delta_s(G)$ 
and $\rho(G):=\min_{s\in V}\rho_s(G)$.  
Since $\pl(G)\le \Delta(G)\le 2\cdot\pl(G)$  and $\pb(G)\le\rho(G)\le 3\cdot\pb(G)$, a path decomposition of length at most $2\cdot\pl(G)$ and of breadth at most $3\cdot\pb(G)$ of an $n$ vertex  graph $G$ can be computed in $O(n^3)$ time~\cite{DrKoLe2017} (by iterating over all start vertices $s\in V$). Furthermore, taking into account Lemma \ref{lm:layer-length}, we get $\adc(G)\le \Delta(G)\le 2\cdot\pl(G)$ and $\adc(G)\le 2\cdot \rho(G)\le 6\cdot\pb(G)$ for every graph $G$. Since, by construction of  $L^+(s, G)$, $\Delta_s(G) \le length(L(s, G))+1$ holds for every graph $G$ and every vertex $s$ of $G$, we have $\Delta(G) \le 3\cdot\adc(G)+3.$  

Combining Proposition  \ref{prop:ld-our}, Lemma  \ref{lm:adc}, Lemma \ref{lm:layer-length}, and Lemma \ref{lm:layer-adc}, we get the following inequalities. 

\begin{theorem} \label{th:adc-pl}
	For every graph $G$, $\pl(G)\le \Delta(G)\le 2\cdot\pl(G)$,    $\pb(G)\le\rho(G)\le 3\cdot\pb(G)$, 
    $$\adc(G)\le \Delta(G)\le    3\cdot \adc(G)+3, $$
  $$\frac{\adc(G)}{2}\leq \frac{\Delta(G)}{2}\leq\pl(G)\leq 2\cdot\adc(G)+1.$$     
\end{theorem}

There is a more general notion of ``quasi-isometry" between graphs. This is a concept from metric spaces, but we define it here just for graphs and paths. Let $G$ be a graph, $P$ be a path (possibly with weights on its edges) and $\psi: V(G)\rightarrow  V(P)$ be a map. Let $L\ge 1$ and $C\ge 0$ be constants. We say that $\psi$ is an $(L,C)$-quasi-isometry if: 
\begin{enumerate}
    \item[(i)] for all $u,v\in V(G)$, $\frac{1}{L}d_G(u,v)-C\le d_P(\psi(u),\psi(v))\le L d_G(u,v)+C$; and  \\
    \item[(ii)] for every $y\in V(P)$ there is $v\in V(G)$ such that $d_P(\psi(v),y)\le C$.  
\end{enumerate}
Here, if $P$ has weights on its edges, $d_P(\psi(u),\psi(v))$ is the sum of weights of all edges of the subpath of $P$ connecting $\psi(u)$ with $\psi(v)$.  

Although $(L,C)$-quasi-isometry from $G$ to $P$ looks very general, we will show that for all $L,C$ there is a $C'$ such that if there is an $(L,C)$-quasi-isometry from $G$ to a path, then there is a $(1,C')$-quasi-isometry from $G$ to a path (see also a similar  result of Kerr \cite{Kerr} stated for quasi-isometry between graphs and trees).  First, we will show that  if there is an $(L,C)$-quasi-isometry from $G$ to a path with some constants $L$ and $C$, then  the path-length of $G$ can be bounded by a constant. By Theorem  \ref{th:adc-pl}, $G$ admits a distance ($2\cdot\pl(G)$)-approximating caterpillar tree. By showing that every caterpillar can be embedded to a path with an additive distortion, we get our result.   

\begin{lemma} \label{lm:quasi-isom-pl}
If there is an $(L,C)$-quasi-isometry from $G$ to a path $P$, then $\pl(G)\le L(L+2C)$
\end{lemma}
\begin{proof} Let $\ell:=L+C$ and $P=(x_0,x_1,\dots,x_q)$. 
Since $P$ can be  weighted, we can assume that for every $x_i$ there is a vertex $v$ in $G$ such that $\psi(v)=x_i$. 

Consider the following path-decomposition of $G$ created from $\psi: V(G)\rightarrow  V(P)$. Let $v$ be an arbitrary vertex of $G$ and let $\psi(v)=x_i$ for some $i\in \{0,\dots,q\}$. Create a bag $X_i$ consisting of all vertices $u$ of $G$ such that $\psi(u)=x_j$, $j\ge i$ and $d_P(\psi(v),\psi(u))\le \ell$. It is easy to show that  the sequence $\{X_0 , \dots , X_q \}$ is a path-decomposition of $G$. Evidently, each vertex $v$ of $G$ belongs to a bag. For every edge $uv$ of $G$, we have $d_P(\psi(v),\psi(u))\le L d_G(u,v)+C=L+C=\ell$. Therefore, assuming $\psi(v)=x_i$, $\psi(u)=x_j$ and $i\le j$, both $u$ and $v$ belong to $X_i$. Let now $i\le j \le k,$ and consider a vertex $v$ of $G$ such that $v\in X_i \cap X_k$. 
We have $d_P(x_i,\psi(v))\le \ell$ and $d_P(x_k,\psi(v))\le \ell$. 
Necessarily, $d_P(x_j,\psi(v))\le \ell$ and, hence, $v\in X_j$. 

To show that the length of path-decomposition $\{X_0 , \dots , X_q \}$ of $G$ is at most $L(L+2C)$, consider any two vertices $u$ and $v$ in $X_i$, $i\in \{0,\dots,q\}$. We have $d_G(u,v)\le L(d_P(\psi(v),\psi(u))+C)$. Since $d_P(\psi(v),\psi(u))\le \ell$, we conclude $d_G(u,v)\le L\ell+LC= L(L+2C)$.
\qed
\end{proof}

A similar result (including also an analog of Theorem   \ref{tm:quasi-isom}) was proven in \cite{BerSey2024} for an $(L,C)$-quasi-isometry from $G$ to a tree and the tree-length of $G$ (see also \cite{DrKoLe2017} for a related result for the minimum line-distortion and the path-length of a graph).  

\begin{theorem}  \label{tm:quasi-isom}
For every graph $G$, the following three statements are equivalent:
\begin{enumerate}
    \item[(i)] the path-length of $G$ is bounded;
    \item[(ii)] there is an $(L,C)$-quasi-isometry to a path with $L,C$ bounded;
    \item[(iii)] there is an $(1,C')$-quasi-isometry to a path with $C'$ bounded. 
\end{enumerate}
\end{theorem}
\begin{proof} 
By Theorem  \ref{th:adc-pl} and Lemma  \ref{lm:quasi-isom-pl}, it remains only to mention that every (unweighted) caterpillar tree  
can straightforwardly be embedded to a path with an additive distortion. 

Let $T=(V,E)$ be a caterpillar with the central path $P=(x_0,\dots,x_q$). Let $\psi: V(T)\rightarrow  V(P)$ be a map that maps every central path vertex to itself and every non-central path vertex to the unique central path vertex to which it is adjacent in $T$. 
Then, $d_T(u,v)-2\le d_P(\psi(u),\psi(v))\le d_T(u,v)$ holds for all $u,v\in V(T)$, and for every $y\in V(P)$ there is $v\in V(T)$ such that $d_P(\psi(v),y)=0$.  
\qed
\end{proof}

\subsection{Powers of AT-free graphs, $k$-dominating pairs and $k$-dominating shortest paths}\label{sec:pat}

An independent set of three vertices such that each pair is joined by a path that avoids the closed neighborhood of the third is called an {\em asteroidal triple}. A graph $G$ is an {\em AT-free graph} if it does not contain any asteroidal triples \cite{AT-free-first}. 
Recall that the $k^{th}$-power of a graph $G = (V, E)$ is a graph $G^k = (V, E')$ such that for every $x, y \in V$ ($x\neq y$), $xy \in E'$ if and only if $d_G(x, y) \le k$. 

We will need the following interesting results from \cite{DrKoLe2017}. 

\begin{proposition} [\cite{DrKoLe2017}] \label{prop:power}  For a graph $G$ with $\pl(G) \le \lambda$, $G^{2\lambda-1}$ is an AT-free graph. Furthermore, if $\pb(G) \le \rho$, then $G^{4\rho-1}$ is AT-free.
\end{proposition}

\begin{proposition} [\cite{DrKoLe2017}] \label{prop:pl(AT)}  If $G$ is an AT-free graph, then $\pb(G)\le \pl(G)\le 2$. Furthermore, there are AT-free graphs with path-breadth equal to 2. 
\end{proposition}

It is known \cite{Ray1987} (see also \cite{genPow} for a more general result)  that if $G^k$ is an AT-free graph then,  $G^{k+1}$ is an AT-free graph as well.  
For a graph $G$, denote by $\pat(G)$ the minimum $k$ such that $G^k$ is an AT-free graph. We call it {\em power-AT index} of a graph $G$. 
Since AT-free graphs can be recognized in $O(n^{2.82})$ time~\cite{KrSpSODA2003}, 
the power-AT index $\pat(G)$ of a graph $G$ can be computed in $O(n^{2.82}\log n)$ time. 

Clearly, $\pat(G)$ is at most the diameter of $G$ for any graph $G$ as a clique is an AT-free graph. Proposition \ref{prop:power} says that  $\pat(G)$ is at most $2\cdot\pl(G)-1$ (and hence at most $4\cdot\pb(G)-1$) for every graph $G$. 
Below, we will show that the power-AT index of a graph $G$ is bounded if and only if the path-length of $G$ is bounded. This links the rich theory behind of AT-free graphs (see \cite{Andreas-book,Chang2003PowersOA,AT-free-first,AT-free-second,AT-free-last,KrSpSODA2003,Ekki-dp}  and papers cited therein) to the path-length. 

We will need an important result from \cite{AT-free-second}. Given a vertex $v$, vertices $x$ and $y$ are called {\em unrelated with respect to vertex} $v$ if there exist a path connecting $x$ with $v$ that is outside $N_G[y]$ and a path connecting $y$ with $v$ that is outside $N_G[x]$. A vertex $v$ of a graph $G$ is called {\em admissible} if no pair of vertices is unrelated with respect to $v$.  

\begin{proposition} [\cite{AT-free-second}] \label{prop:admissible}  Every AT-free graph $G$ has an admissible vertex. Futhermore, an admissible vertex of an AT-free graph can be found in linear time.  
\end{proposition}

Using this result from \cite{AT-free-second}, we can prove the following lemma.  

\begin{lemma}  \label{lem:pat-length}
Let $G$ be a graph with $\pat(G)=k$, $s$ be an admissible vertex of AT-free graph $G^k$, and $L(s, G)$ be a layering of $G$ with respect to start vertex $s$.  Then, $length(L(s, G))\le 2k$. Consequently, $\adc(G)\le 2\cdot\pat(G)$ and $\Delta(G)\le 2\cdot\pat(G)+1$. 
\end{lemma}  
\begin{proof} Let $L(s, G)=\{L_0,L_1,\dots, L_q\}$, $L_i = \{u \in V : d_G(s, u) = i\}$, $q=\max\{d_G(s,v): v\in V\}$. Consider an arbitrary layer $L_i$ and arbitrary two vertices $x,y\in L_i$. Since $s$ is admissible in $G^k$, $N_{G^k}[x]$ must intercept all paths of $G^k$ connecting $y$ with $s$ or $N_{G^k}[y]$ must intercept all paths  of $G^k$ connecting $x$ with $s$. Without loss of generality, let $N_{G^k}[y]$  intercept all paths of $G^k$ between $x$ and $s$. 
Consider a shortest path $P$ connecting $x$ with $s$ in $G$. It is also a path in $G^k$. Therefore, there must exist a vertex $w\in P$ such that $d_{G^k}(y,w)\le 1$. Necessarily, $d_{G}(y,w)\le k$. We have $d_G(s,y)\le d_G(s,w)+d_G(w,y)\le d_G(s,w)+k$ and $d_G(s,x)= d_G(s,w)+d_G(w,x).$ Since $d_G(s,x)=i=d_G(s,y)$, we get $d_G(w,x)\le k$. The latter gives $d_G(x,y)\le d_G(x,w)+d_G(w,y)\le 2k$. 

Now, by Lemma \ref{lm:layer-length},  $\adc(G)\le length(L(s, G))\le 2k$ and,  
by construction of $L^+(s, G)$, $\Delta(G)\le \Delta_s(G) \le length(L(s, G))+1\le 2k+1$. 
\qed
\end{proof}

A {\em layout} of a graph $G=(V,E)$ is a linear ordering $\sigma:=\{v_1,v_2,\dots,v_n\}$ of the vertex set of $G$. We  simply write $v_i<v_j$ when $i<j$. A layout is called a \ccp-layout if there exists no triple $x,y,z\in V$ with $x<y<z$ such that $xy\notin E$, $yz\notin E$, $xz\in E$. A graph is called {\em cocomparability} if it possesses a \ccp-layout~\cite{Andreas-book,GolBook}. 
It is known \cite{Chang2003PowersOA}  that if $G$ is an AT-free graph then, for every $k>1$, $G^k$ is a cocomparability graph. Furthermore, if $G^k$ is a cocomparability graph, then  $G^{k+1}$ is also a cocomparability graph~\cite{Chang2003PowersOA,damaschke} (see also \cite{genPow} for some more general results). 
For a graph $G$, denote by $\pcc(G)$ the minimum $k$ such that $G^k$ is a cocomparability graph. We
call it {\em power-cocomparability index} of a graph $G$. 
Cocomparability graphs can be recognized in $O(M(n))$ time~\cite{Spinrad-comp}, 
where $M(n)$ denotes the time complexity of multiplying two $n$ by $n$ matrices of integers (currently is known to be $O(n^{2.371339})$ time \cite{ne-mm-algo}).
Hence, the power-cocomparability index $\pcc(G)$ of a graph $G$ can be computed in $O(M(n)\log n)$ time (see \cite{Chang2003PowersOA}). 
Since every cocomparability graph is AT-free~\cite{AT-free-first}, $\pat(G)\le \pcc(G)$ for every graph $G$. Since $G^2$ is a cocomparability graph for an AT-free graph $G$, we also have  $\pcc(G)\le 2\cdot\pat(G)$.



We do not know if $G^{k+1}$ is a cocomparability graph for a graph $G$ whose $k^{th}$ power $G^k$ is an AT-free graph. Nevertheless, we can complement Proposition  \ref{prop:power} 
by the following lemma. 

\begin{lemma} \label{cpp-pl} For a graph $G$ with $\pl(G)= \lambda$, $G^{2\lambda}$ is a cocomparability graph. 
Consequently, $\pcc(G)\le 2\pl(G)\le 4\pb(G). $
\end{lemma}

\begin{proof} 
Let $\cP(G)=\{X_1 , \dots , X_q \}$ be a path-decomposition of $G$ of length $\lambda$. For each vertex $v$ of $G$, let $\ell(v):=\min\{i: v\in X_i\}$ be the index of the leftmost bag in $\cP(G)=\{X_1 , \dots , X_q \}$ containing $v$.   We can get a layout $\sigma:=\{v_1,v_2,\dots,v_n\}$ of $G$ by ordering vertices with respect to values $\ell(v)$, breaking ties arbitrarily. We show that $\sigma$ is a \ccp-layout of $G^{2\lambda}$, hence proving that $G^{2\lambda}$ is a cocomparability graph.  

Consider arbitrary three vertices $x$, $y$ and $z$  of $G$ with $x<y<z$ in $\sigma$. Assume $x$ and $z$ are adjacent in $G^{2\lambda}$, i.e., $d_G(x,z)\le 2\lambda$. We need to show that $d_G(x,y)\le 2\lambda$ or $d_G(y,z)\le 2\lambda$. Without loss of generality, we can assume that  $d_G(x,y)> \lambda$ and $d_G(y,z)> \lambda$ (otherwise, we are done). Since $x$ and $y$ cannot belong to a common bag (recall that the length of each bag is at most $\lambda$), all bags of $\cP(G)$ containing $x$ are 
strictly to the left in $\cP(G)$ of all bags containing $y$. Similarly,  all bags of $\cP(G)$ containing $y$ are 
strictly to the left in $\cP(G)$ of all bags containing $z$. By properties of path-decompositions  (see \cite{Diestel-book,RobSey1983}), any path $Q$ of $G$ between $x$ and $z$ has a vertex in every bag of $\cP(G)$ containing $y$. Consequently, a shortest path $P$ of $G$ connecting $x$ with $z$ has a vertex $v$ with $d_G(v,y)\le \lambda$. 
Assuming $d_G(x,v)\le d_G(v,z)$,  we get $d_G(x,v)\le d_G(x,z)/2\le 2\lambda/2=\lambda$. Consequently, $d_G(y,x)\le d_G(y,v)+d_G(v,x)\le \lambda+\lambda=2\lambda$.  Hence, if $xz$ is an edge in $G^{2\lambda}$, then $xy$ or $yz$ is an edge in $G^{2\lambda}$, i.e., $\sigma$ is a \ccp-layout of $G^{2\lambda}$.~\qed
\end{proof}

Note that there are graphs with path-length at most 2 whose powers $G^\mu$ are not interval graphs for any constant $\mu\ge 1$. Consider a long ladder (2 by $k$ rectilinear grid, i.e., a chain of $k$ induced $C_4$s for a sufficiently large $k$). It has tree-length 2 and its power $G^k$ contains an induced $C_4$.  

A path $P$ of a graph $G$ is
called $k$-dominating path of $G$ if every vertex $v$ of $G$ is at distance at most $k$ from a vertex of $P$, i.e.,
$d_G(v, P )\le k.$ A pair of vertices $x, y$ of $G$ is called a $k$-dominating pair if every path between $x$ and
$y$ is a $k$-dominating path of $G$. Denote by $\dpr(G)$ (by $\dsp(G)$) the minimum $k$ such that $G$ has a $k$-dominating pair (a $k$-dominating shortest path, respectively). We call $\dpr(G)$ 
($\dsp(G)$) the {\em dominating-pair--radius} (the {\em dominating-shortest-path--radius}, respectively) of $G$.  It is known that every AT-free graph has a 1-dominating pair \cite{AT-free-first}. Using this and a few auxiliary results, we can show that all these four parameters $\pat(G)$, $\pcc(G)$, $\dpr(G)$, $\dsp(G)$ are coarsely equivalent to $\pl(G)$ (and hence to $\pb(G))$.  First, we prove the following lemma. 

\begin{lemma}  \label{lem:pl-pat}  	For every graph $G$, $\pl(G)\le 2\cdot\pat(G)$. 
\end{lemma}
\begin{proof} Let $k:=\pat(G)$. Since $G^k$ is an AT-free graph, by Proposition \ref{prop:pl(AT)}, there is a path-decomposition $\cP$ of $G^k$ of length at most 2. Clearly, it is a path-decomposition of $G$ as well. Furthermore, for any two vertices $u,v$ belonging to same bag of $\cP$, we have $d_{G^k}(u,v)\le 2$ and hence, by the definition of the $k^{th}$-power, $d_G(u,v)\le 2k$. Consequently, $\pl(G)\le 2\cdot\pat(G)$.~\qed
\end{proof}

 Theorem  \ref{th:adc-pl} and Proposition \ref{prop:power}  already imply $\pat(G)\le 2\cdot\pl(G)-1\le 4\cdot\adc(G)+1$.  Using cocomparability graphs, we can improve the upper bound in $\pat(G)\le 4\cdot\adc(G)+1$ to $\pat(G)\le 3\cdot\adc(G)+2$ (clearly, this is an improvement for all $G$ with $\adc(G)>0$, i.e., when $G$ itself is not a caterpillar).  

\begin{lemma}  \label{lem:adc-pat}  	For every graph $G=(V,E)$, $\pat(G)\le\pcc(G)\le 3\cdot\adc(G)+2$. 
\end{lemma}
\begin{proof} We need to prove only the second inequality. Let $\adc(G)=\delta$, i.e., there is a caterpillar tree $T=(V,E')$ such that $|d_G(u,v)-d_T(u,v)|\le \delta$ holds for every $u,v\in V$. Let $P=\{x_0,x_1,\dots,x_t\}$  be the central path of $T$. We can get a layout $\sigma=\{v_1,v_2,\dots,v_n\}$ of the vertex set of $G$ from $P$ as follows.  For every $i=0,\dots,t-1$, we squeeze  all vertices adjacent to $x_i$ in $T$, that are not in $P$, between $x_i$ and $x_{i+1}$, in arbitrary order. All vertices adjacent to $x_t$ in $T$, that are different from $x_{t-1}$, we place after $x_t$, again in arbitrary order.  We show that $\sigma$ obtained this way is a \ccp-layout of $G^\mu:=(V,E^\mu)$ for $\mu=3\delta+2$. 

Assume there exists a triple $x,y,z\in V$ with $x<y<z$ such that $xy\notin E^\mu$, $yz\notin E^\mu$, $xz\in E^\mu$. 
We have $d_G(x,z)\le \mu$ and $d_G(y,z)>\mu$, $d_G(x,y)> \mu$. Let $x\in N_T[x_i]$, $y\in N_T[x_j]$, $z\in N_T[x_k]$ for $x_i,x_j,x_k\in P$. Since 
$|d_G(u,v)-d_T(u,v)|\le \delta$ holds for every $u,v\in V$, we get  $d_T(x,z)\le d_G(x,z)+\delta\le \mu+\delta=4\delta+2$, $d_G(x,x_i)\le d_T(x,x_i)+\delta\le 1+\delta$, $d_G(y,x_j)\le 1+\delta$, $d_G(z,x_k)\le 1+\delta$. As $d_G(y,z)>\mu$, $d_G(x,y)> \mu$ and $x<y<z$, necessarily, $i<j<k$ holds (if, for example, $i=j$, then $d_G(x,y)\le d_G(x,x_i)+ d_G(y,x_j)\le 2\delta+2\le \mu$, and a contradiction arises). 

From $4\delta+2\ge d_T(x,z)=d_T(x,x_j)+d_T(x_j,z)$ we get $d_T(x,x_j)\le 2\delta+1$ or $d_T(x_j,z)\le 2\delta+1$. Assume, without loss of generality, that $d_T(x,x_j)\le 2\delta+1$ holds. Then, $d_T(x,y)\le d_T(x,x_j)+1\le 2\delta+2$, implying  
$d_G(x,y)\le d_T(x,y)+\delta\le  3\delta+2=\mu$. The latter contradicts with $d_G(x,y)>\mu$.~\qed
\end{proof}

Our next lemma shows that every graph having a $k$-dominating shortest path can be embedded to a caterpillar tree with an additive distortion at most $2k+2$.  

\begin{lemma}  \label{lem:adc-dsp-upperb}  	For every graph $G=(V,E)$, $\adc(G)\le 2\cdot \dsp(G)+2$. 
\end{lemma}
\begin{proof} 
%
Let $P$ be a $k$-dominating shortest path of $G$. For every vertex $v$ of $G$, denote by $x_v$ a vertex of $P$ closest to $v$, i.e., with $d_G(v,x_v)= d_G(v,P)\le k$. We can build a caterpillar tree $T$ from $G$ by using $P$ as the central path of $T$ and making every vertex $v$ not in $P$ adjacent in $T$ to $x_v\in P$. We show that, for every pair $x,y\in V$, $d_T(x,y)-2-2k\le d_G(x,y)\le d_T(x,y)+2k$. 
Indeed, since $P$ is a shortest path of $G$, by the triangle inequality, we get $$d_T(x,y)\le d_P(x_v,y_v)+2=  d_G(x_v,y_v)+2\le d_G(x,x_v)+d_G(x,y)+d_G(y,y_v)+2\le d_G(x,y)+2k+2.$$ We also get $$d_G(x,y)\le d_G(x,x_v)+d_G(x_v,y_v) +d_G(y_v,y)\le d_P(x_v,y_v)+2k=d_T(x_v,y_v)+2k\le d_T(x,y)+2k.$$ 
Thus, $\adc(G)\le 2\cdot\dsp(G)+2$. 
~\qed
\end{proof}

Now, we can prove the following theorem. 

\begin{theorem} \label{th:pat-pl}
	For every graph $G$,  
    $\rho(G)\le 2\cdot\dsp(G)+1$, 
    $\dsp(G)\le  \dpr(G)\le 3\cdot\dsp(G)$, and 
$$\dsp(G)\leq \dpr(G)\le \pl(G)\le 4\cdot\dsp(G)\le 4\cdot\dpr(G),$$ 
$$\dsp(G)\leq \dpr(G)\le \Delta(G)\le 4\cdot\dsp(G)+1\le 4\cdot\dpr(G)+1,$$    
$$\frac{\adc(G)-2}{2}\leq \dsp(G)\leq \dpr(G)\le  2\cdot\adc(G)+1,$$  
$$\frac{\adc(G)}{2}\leq  \pat(G)\le \pcc(G)\le 3\cdot\adc(G)+2,$$ 
$$\frac{\pl(G)}{2}\leq \pat(G)\le 2\cdot\pl(G)-1,$$
$$\frac{\Delta(G)-1}{2}\le  \pat(G)\le 2\cdot\Delta(G)-1,$$
$$\pat(G)\le \pcc(G)\le 2\cdot\pat(G),$$ 
$$\frac{\pl(G)}{2}\leq \pat(G)\le \pcc(G)\le 2\cdot \pl(G).$$  
\end{theorem}

\begin{proof} Clearly, $\dsp(G)\leq \dpr(G)$ and $\pl(G)\le \Delta(G)$ 
for every graph $G$. The inequality $\pat(G)\le 2\cdot\pl(G)-1$ follows from Proposition \ref{prop:power}. The inequalities 
$\adc(G)\le 2\cdot\pat(G)$, $\Delta(G)\le 2\cdot\pat(G)+1$,  
$\pcc(G)\le 2\cdot \pl(G)$, $\pl(G)\le 2\cdot\pat(G)$,  $\pat(G)\le\pcc(G)\le 3\cdot\adc(G)+2$ and $\adc(G)\le 2\cdot \dsp(G)+2$ are given by Lemma  \ref{lem:pat-length}, Lemma \ref{cpp-pl}, Lemma  \ref{lem:pl-pat}, Lemma \ref{lem:adc-pat} and Lemma \ref{lem:adc-dsp-upperb}, respectively.   
%
The inequality $\pl(G)\le 4\cdot\dsp(G)$ follows from \cite{Dulei2019}. Recall also that $\pcc(G)\le 2\cdot\pat(G)$ holds. 

In any path-decomposition $\cP(G) = \{X_1, X_2, \dots , X_q \}$  of $G$ of length $\lambda$, for any two vertices $x \in X_1$ and $y \in  X_q$, any path $P$ between $x$ and $y$, by properties of
path-decompositions (see \cite{Diestel-book,RobSey1983}), has a vertex in every
bag of $\cP(G)$. Consequently, as each vertex $v$ of $G$ belongs to some bag $X_i$ of $\cP(G)$, there is a vertex $u \in P$
with $u \in X_i$. The latter implies $d_G(v, u) \le \lambda$. Hence, for every graph $G$, $\dpr(G)\leq \pl(G)\le\Delta(G)$.  By Lemma \ref{lm:adc}, we also get  $\dpr(G)\le  \pl(G)\le 2\cdot\adc(G)+1$.

Let $G$ be a graph such that $G^k$ is an AT-free graph. Since every AT-free graph has a 1-dominating pair (see \cite{AT-free-first}), $G^k$ must have a 1-dominating pair $x,y$. Every path $P$ connecting vertices $x$ and $y$ in $G$ is a path of $G^k$, too. Hence, for every vertex $v$ of $G$, we have $d_{G^k}(v,P)\le 1$, implying $d_{G}(v,P)\le k$. Thus, $x,y$ is a $k$-dominating pair of $G$, i.e., $\dpr(G)\le k$. This shows $\dpr(G)\le \pat(G)$. 

Let now, in what follows,  $P$ be a $k$-dominating shortest path of $G$  connecting vertices $s$ and $t$ such that $k=\dsp(G)$. Consider layering $L(s, G):=\{L_0,L_1,\dots L_q\}$ of $G$, where $L_i = \{u \in V : d_G(s, u) = i\}$, $q=\max\{d_G(u,s): u\in V\}$. It is easy to show that $breadth(L(s, G))$ is at most $2k$, i.e., for every $i\in\{1,\dots,q\}$ there is a vertex $v_i$ in $G$ with $d_G(u, v_i)\le 2k$ for every $u\in L_i$. 
Indeed, let $v\in L_i$ be an arbitrary vertex of $G$ and $p(v)$ be a vertex of $P$ such that $d_G(v,p(v))\le k$. Let also $v_i$ be a vertex of $P\cap L_i$. It exists if $i\le d_G(s,t)$. If $i> d_G(s,t)$, set $v_i=t$. Since $P$ is a shortest path of $G$, $|d_G(s,p(v)-d_G(s,v_i)|\le k$. Therefore, $d_G(v,v_i)\le d_G(v,p(v))+d_G(v_i,p(v))\le 2k$. 

Since $breadth(L(s, G))\le 2k$, $\rho(G)\le \rho_s(G)\le breadth(L(s, G))+1$, $length(L(s, G))\le 2\cdot breadth(L(s, G))$, and $\Delta(G)\le \Delta_s(G)\le length(L(s, G))+1$, we have  $\rho(G)\le 2k+1=2\cdot\dsp(G)+1$ and  $\Delta(G)\le 4\cdot\dsp(G)+1$. Applying Lemma   \ref{lm:layer-length} and taking into account $breadth(L(s, G))\le 2k$, we conclude also $\adc(G)\le 4k=4\cdot\dsp(G)$ (but this bound is not better than the bound established in Lemma  \ref{lem:adc-dsp-upperb}, when $\dsp(G)>0$).    

  \begin{figure}[htb]
    \begin{center} 
      \begin{minipage}[b]{15cm}
        \begin{center} 
          \vspace*{-12mm}
          \includegraphics[height=15cm]{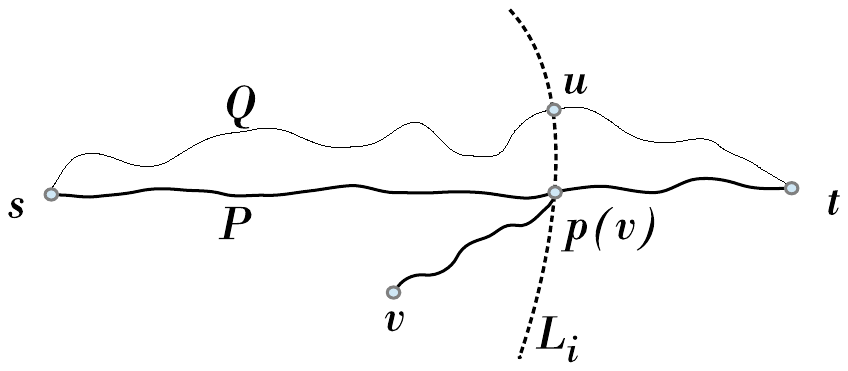}
        \end{center} \vspace*{-120mm}
        \caption{\label{fig:one} Illustration to the proof of Theorem \ref{th:pat-pl} (part for  $\dpr(G)\le 3\cdot\dsp(G)$). }  
      \end{minipage}
    \end{center}
   \vspace*{-5mm}
  \end{figure}

Consider now, additionally to $k$-dominating shortest path $P$ between $s$ and $t$, an arbitrary path $Q$ connecting $s$ and $t$ in $G$. Let $v$ be an arbitrary vertex of $G$ and assume $p(v)$ is a vertex of $P$ with $d_G(v,p(v))\le k$. Assume $p(v)$ belongs to layer $L_i$ of $L(s, G):=\{L_0,L_1,\dots L_q\}$ (see Figure \ref{fig:one}).   Since $Q$ must intersect $L_i$, we have a vertex $u\in L_i\cap Q$. From the proof above about $breadth(L(s, G))\le 2k$, for $v_i=p(v)$, $d_G(v_i,u)\le 2k$ holds. Hence, we get $d_G(u,v)\le d_G(u,v_i)+d_G(v_i,v)\le 3k$. Thus, an arbitrary path $Q$ connecting $s$ and $t$ has every vertex $v$ of $G$ within distance at most $3k$. The latter implies that $s$ and $t$ form  a $3k$-dominating pair of $G$, i.e., $\dpr(G)\le 3\cdot\dsp(G)$. 
\qed
\end{proof}


For relations between $\pat(G)$, $\pcc(G)$ and $\dsp(G)$, $\dpr(G)$, $\Delta(G)$,  
we get the following. 

\begin{corollary}  \label{cor:ineq-pl-pat-dsp}  
For every graph $G$, 
 $$\dsp(G)\le \dpr(G)\le \pat(G)\le  \pcc(G)\le   6\cdot \dsp(G)\le  6\cdot \dpr(G),$$  
$$ \frac{\Delta(G)-1}{2}\le   \pcc(G)\le 2\cdot\Delta(G).$$ 
\end{corollary}

\begin{proof} 
We get $\pcc(G)\le 2\cdot\Delta(G)$ from $\pcc(G)\le  2\cdot\pl(G)$ and $\pl(G)\le\Delta(G)$.  
We get $\frac{\Delta(G)-1}{2}\le  \pcc(G)$ from $\frac{\Delta(G)-1}{2}\le  \pat(G)$ and   $\pat(G)\le\pcc(G)$. From $\pcc(G)\le  3\cdot\adc(G)+2$ and $\adc(G)\le  2\cdot\dsp(G)+2$, we get also  $\pcc(G)\le  6\cdot\dsp(G)+8$. 
Using a direct proof, we can improve that to $\pcc(G)\le  6\cdot\dsp(G)$.

Let $k:=\dsp(G)$ and $P=(x_0,x_1,\dots,x_q)$ be a $k$-dominating shortest path of $G$. Consider a $BFS(P)$-tree $T$ of $G$, i.e., a breadth first search tree starting from path $P$. This tree $T$ has $P=\{x_0,x_1,\dots,x_q\}$  as a central path and has branches stretching out from each $x_i\in P$ (see Figure \ref{fig:two}(a)). Let $T_i$ be all branches of $T$ rooted at $x_i$. We know, for every vertex $v$ in $T_i$, $d_G(v,x_i)\le k$. We can get a  layout $\sigma=\{v_1,v_2,\dots,v_n\}$ of the vertex set of $G$ from $P$ and $T_i$s as follows.  For every $i=0,\dots,q-1$, we squeeze  all vertices of $T_i$, different from $x_i$, between $x_i$ and $x_{i+1}$, in arbitrary order. All vertices of $T_q$, different from $x_{q}$, we place after $x_q$, again in arbitrary order. We show that $\sigma$ obtained this way is a \ccp-layout of $G^\mu:=(V,E^\mu)$ for $\mu=6k$. 

Assume there exists a triple $x,y,z\in V$ with $x<y<z$ in $\sigma$ such that $xy\notin E^\mu$, $yz\notin E^\mu$, $xz\in E^\mu$. 
We have $d_G(x,z)\le \mu$ and $d_G(y,z)>\mu$, $d_G(x,y)> \mu$. Let $x\in T_i$, $y\in  T_j$, $z\in T_\ell$. Since $d_G(y,z)>\mu$, $d_G(x,y)> \mu$ and, for every two vertices $u,v$ in $T_\nu$ ($\nu\in \{0,\dots,q\}$), $d_G(u,v)\le d_G(u,x_\nu)+d_G(v,x_\nu)\le 2k\le \mu$,  we get $i<j<\ell$ in $P$ and $x_i\le x<x_{i+1}\le x_j\le y<x_{j+1}\le x_\ell\le z$ in $\sigma$ (see Figure \ref{fig:two}(b)). Since $d_G(x,z)\le \mu$ and $P$ is a shortest path of $G$, we also have $d_P(x_i,x_\ell)=d_G(x_i,x_\ell)\le d_G(x_i,x)+ d_G(x,z)+ d_G(z,x_\ell)\le \mu+2k.$ 
Assuming, without loss of generality, $d_P(x_i,x_j)\le d_P(x_j,x_\ell)$, we get $d_P(x_i,x_j)\le \mu/2+k$. The latter implies $d_G(x,y)\le d_G(x,x_i)+d_P(x_i,x_j) +d_G(x_j,y)\le k+(\mu/2+k)+k=\mu$, and a contradiction with $d_G(x,y)>\mu$ arises.~\qed    
\end{proof}

  \begin{figure}[htb]
    \begin{center} 
      \begin{minipage}[b]{15cm}
        \begin{center} 
          \vspace*{-16mm}
          \includegraphics[height=16cm]{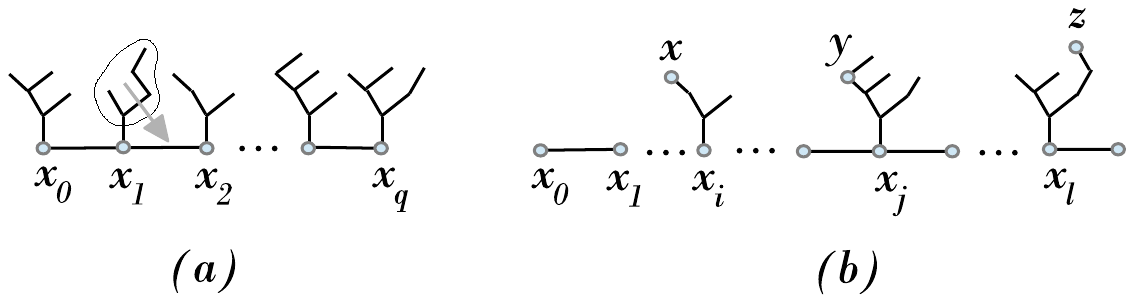}
        \end{center} \vspace*{-124mm}
        \caption{\label{fig:two} Illustrations to the proof of Corollary \ref{cor:ineq-pl-pat-dsp}. }  
      \end{minipage}
    \end{center}
   \vspace*{-5mm}
  \end{figure}

We can rephrase a part of Corollary  \ref{cor:ineq-pl-pat-dsp} and a part of Theorem \ref{th:pat-pl} in the following ways (which are of independent interest). It would be interesting to see if those constants can be improved. 

\begin{corollary}  \label{cor:conseq-AT-powers}
$(a)$ If a graph $G$ has a $k$-dominating shortest path, then $G^\mu$ is a cocomparability (and, hence, an AT-free) graph for some $\mu\le 6k$.   \\ 
$(b)$ If a graph $G$ can be embedded to a caterpillar tree with an additive distortion $\delta$, then $G^\mu$ is a cocomparability  $($and, hence, an AT-free$)$ graph for some $\mu\le 3\delta+2$.   
\end{corollary}



It is easy to see that a pair of vertices $x$ and $y$ is a $k$-dominating pair if and only if, for every
vertex $w \in V \setminus (D_G(x, k) \cup D_G(y, k))$, the disk $D_G(w, k)$ separates $x$ and $y$. Hence, a $k$-dominating
pair, with minimum $k$ ($k=\dpr(G)$), of an arbitrary graph $G = (V, E)$ with $n$ vertices can be found in $O(n^3 \log n)$ time (see \cite{DrKoLe2017} for details). 
It is not very likely that there is a linear time algorithm to find a 1-dominating pair (called just a {\em dominating pair}), if it exists,
since it is shown in \cite{KrSpSODA2003} that finding a dominating pair is essentially as hard as finding a triangle in
a graph. Yet, since path-decompositions with small length are closely related to $k$-domination, one
can search for $k$-dominating pairs in dependence of the path-length of a graph. We do not know
how to find in linear time for an arbitrary graph $G$ a $k$-dominating pair with $k\le \pl(G)$. However, as it was shown in \cite{DrKoLe2017}, there is a linear time algorithm that determines a $k$-dominating pair of an arbitrary graph $G$ such that $k \le 2\cdot \pl(G)$. Graphs that have a 1-dominating pair  were extensively studied in \cite{DeoKr-dp,Ekki-dp}. 

There is also \cite{DrKoLe2017} a very simple linear time algorithm that for an arbitrary graph $G$ finds a $k$-dominating shortest path of $G$ with $k\le 2\cdot\pl(G)$. The problem of finding in a graph $G$ a $k$-dominating shortest path with minimum $k$ ($k=\dsp(G)$) is known in literature under the name the {\em Minimum Eccentricity Shortest Path} (MESP) problem (see \cite{DrLe2016,DrLe2017}). It is known  that the MESP problem is NP-hard even for planar bipartite graphs with maximum degree 3 and is $W[2]$-hard for general graphs \cite{DrLe2017}. A 2-approximation, a 3-approximation, and an 8-approximation for the MESP problem can be computed in $O(n^3)$ time, in $O(nm)$ time, and in linear time, respectively  \cite{DrLe2017} (see also \cite{BiMoPl2016} for some improvements in the complexities). Polynomial time algorithms
for the MESP problem are known for chordal graphs and for distance hereditary graphs \cite{DrLe2016}. Fixed parameter
tractability of the MESP problem with respect to various graph parameters like modular width, distance
to cluster graph, maximum leaf number, feedback edge set, etc. have been also studied recently \cite{BhJaKaSaVe2023,KuSu2003}. All these efficient algorithms for finding a $k$-dominating shortest path, with small $k$, of a graph $G$ can be used for  approximation of other graph parameters that are coarsely equivalent to $\dsp(G)$. Relation of the  MESP problem with other problems like the minimum distortion embedding on a line \cite{DrLe2017} and $k$-laminar problem \cite{BiMoPl2016}  have also been established. 



\subsection{McCarty index}\label{sec:mci}

The following result from \cite{DrKoLe2017} generalizes a characteristic property of the famous class of AT-free graphs. 

\begin{proposition} [\cite{DrKoLe2017}] \label{prop:mci} Let $G$ be a graph with $\pl(G) \le \lambda$. Then, for every three vertices $u, v, w$ of $G$, one of those vertices, say v, satisfies the following property: the disk of radius $\lambda$ centered at $v$ intercepts every path connecting $u$
and $w$, i.e., after the removal of the disk $D_G(v, \lambda)$ from $G$, $u$ and $w$ are not in the same connected
component of $G \setminus D_G(v,\lambda)$. 
\end{proposition}

We can introduce a new parameter {\em McCarty index} of a graph $G$, denoted by $\mci(G)$. This is the minimum $r$ such that for every triple of vertices $u,v,w$ of $G$, disk of radius $r$ centered at one of them intercepts all paths connecting two others. 
In \cite{BerSey2024}, the {\em McCarty-width}  of a graph was introduced and it was shown that the McCarty-width of a graph $G$ is small if and only if the tree-length of $G$ is small (i.e., these two parameters are coarsely equivalent). Recall \cite{BerSey2024} that a graph $G$ has {\em McCarty-width} $r$ if $r\ge 0$ is minimum such that the following holds: for every three vertices $u, v, w$ of $G$, there is a vertex $x$ such that no connected component of $G\setminus D_G(x,r)$  contains two of $u, v, w$. Our McCarty index is a linearization of McCarty-width, it forces one of $u,v,w$ to play the role of $x$.  
It is not hard to see that the McCarty index $\mci(G)$ of a graph $G$ can be computed in at most $O(n^3)$ time. 

We will show that the McCarty index $\mci(G)$ of a graph $G$ is coarsely equivalent to the path-length $\pl(G)$ of $G$. By Proposition \ref{prop:mci}, we have $\mci(G)\le \pl(G).$ Our next result shows that $\pat(G)\le 2\cdot \mci(G)-1.$ 

\begin{lemma}  \label{lem:mci-pat}
For every graph $G$,  $\pat(G)\le 2\cdot \mci(G)-1.$ 
\end{lemma}
\begin{proof} Let $r:=\mci(G)$. We show that $G^{2r-1}$ is an AT-free graph. 

Consider three arbitrary vertices $u,v,w$ that form an independent set in $G^{2r-1}$, i.e., $d_G(x,y)> 2r-1$ for each $x,y\in \{u,v,w\}$. Since  $r:=\mci(G)$, without loss of generality, let $D_G(w,r)$ intercept every path connecting $u$ and $v$ in $G$. 
Consider an arbitrary path $P$ of $G^{2r-1}$ connecting $u$ and $v$.  
For every edge $xy$ of $P$, consider in $G$ a path $P_G(x,y)$ between $x$ and $y$ whose length in $G$ is at most $2r-1$. We can get a path $P^+\supseteq P$ of $G$ between $u$ and $v$ by replacing every edge $xy$ of $P$ with path $P_G(x,y)$. Since every subpath  $P_G(x,y)$ of $P^+$ has length at most $2r-1$, every vertex of $P^+$ has a vertex of $P$ within distance at most $r-1$ in $G$. 
We know that $D_G(w,r)$ intercepts every path connecting $u$ and $v$ in $G$. Hence, there must exist a vertex $z\in P^+\cap D_G(w,r)$. Let $p(z)$ be a vertex of $P$ that is in $G$ within distance at most $r-1$ from $z$. By the triangle inequality,  $d_G(w,p(z))\le d_G(w,z)+d_G(z,p(z))\le r+r-1=2r-1$. Thus, every path $P$ of $G^{2r-1}$ connecting $u$ and $v$ has a vertex in $D_G(w,2 r-1)=D_{G^{2r-1}}(w,1)$.  By definition, $u,v,w$ cannot form an asteroidal triple in $G^{2r-1}$. 

Thus, $G^{2r-1}$ does not contain any asteroidal triples, i.e., 
$G^{2r-1}$ is an AT-free graph.   
\qed
\end{proof}

From  Lemma \ref{lem:pat-length} and Lemma  \ref{lem:mci-pat}, we immediately get. 
\begin{corollary}  \label{cor:more-ineq-mci}
For every graph $G$, 
$\adc(G)\le 4\cdot\mci(G)-2$ and $\Delta(G)\le 4\cdot\mci(G)$.
\end{corollary}

Combining Proposition \ref{prop:mci}, Lemma \ref{lem:mci-pat}  and Corollary \ref{cor:more-ineq-mci} with Lemma  \ref{lem:pl-pat}, we obtain (recall also $\mci(G)\le \pl(G)\le 2\cdot \adc(G)+1$ and $\pl(G)\le \Delta(G)$).  
\begin{theorem} \label{th:pat-pl-mci}
	For every graph $G$,    
  $$\mci(G)\le \pl(G)\le 2\cdot\pat(G)\le 4\cdot\mci(G)-2,$$
  $$\mci(G)\le \Delta(G)\le 4\cdot\mci(G),$$ 
  $$\frac{\mci(G)-1}{2}\le \adc(G)\le 4\cdot\mci(G)-2.$$ 
\end{theorem} 


Combining Theorem  \ref{th:pat-pl-mci} with previous inequalities, we also get. 
\begin{corollary}  \label{cor:ineq-with-mci}
For every graph $G$, 
$$\frac{\mci(G)}{4}\le \dsp(G)\le \dpr(G)\le 2\cdot\mci(G)-1,$$
$$\frac{\mci(G)}{2}\le \pcc(G)\le  4\cdot\mci(G)-2.$$
\end{corollary}
\begin{proof} 
For $\dsp(G)$, $\dpr(G)$, we have 
$\frac{\mci(G)}{4}\le \frac{\pl(G)}{4}\le \dsp(G)\le \dpr(G)\le \pat(G)\le 2\cdot\mci(G)-1.$ 
For $\pcc(G)$,  we have 
$\frac{\mci(G)}{2}\le \pat(G)\le \pcc(G)\le 2\cdot\pat(G)\le 4\cdot\mci(G)-2.$ 
\qed
\end{proof}

From Theorem \ref{th:pat-pl-mci} and Corollary  \ref{cor:ineq-with-mci}, we can extract also the following results of independent interest. 
It would be interesting to get direct (constructive) proofs for (b) and (c), possibly also improving constants; part (a) is addressed in  Corollary \ref{cor:more-ineq-mci}.
\begin{corollary}  \label{cor:conseq-mci-dpr}
$(a)$ If the McCarty index of $G$ is $r$, then $G$ has a $(2r-1)$-dominating pair and $G$ can be embedded to a caterpillar tree with an additive distortion at most $4r-2$. \\ 
$(b)$ If a graph $G$ has a $k$-dominating shortest path, then 
the McCarty index of $G$ is at most $4k$. \\ 
$(c)$ If a graph $G$ can be embedded to a caterpillar tree with an additive distortion $\delta$, then the McCarty index of $G$ is at most $2\delta+1$.   
\end{corollary}


\subsection{$K$-Fat $K_3$-minors and $K$-fat $K_{1,3}$-minors}\label{sec:fat-minors}
It already follows from the results of \cite{Diestel++,GeorPapa2023} that the path-length of a graph $G$ is bounded by a constant if and only if $G$ has neither $K$-fat $K_3$-minor nor $K$-fat $K_{1,3}$-minor for some constant $K>0$. This provides a coarse characterization of path-length through forbidden $K$-Fat ($K_3, K_{1,3}$)-minors. 
Here, we give an alternative simple proof, with precise constants, of this result.   

Although a $K$-fat $H$-minor can be defined with any graph $H$ \cite{GeorPapa2023}\footnote{Fat minors were first introduced in \cite{ChepoiDNRV12} under the name {\em metric minors} and then independently in \cite{FuPapa2021} under the name {\em fat minors}.}, here we give a definition only for $H=K_3$ and $H=K_{1,3}$ as we work in this paper only with these minors.  
It is said that a graph $G$ has a {\em $K$-fat $K_3$-minor} ($K>0$) if there are three connected subgraphs $H_1$, $H_2$, $H_3$ and three simple paths $P_{1,2}$, $P_{2,3}$, $P_{1,3}$ in $G$ such that for each $i,j\in \{1,2,3\}$ ($i\neq j$), $P_{i,j}$ has one end in $H_i$ and the other end in $H_j$ and  $|P_{i,j}\cap V(H_i)|=|P_{i,j}\cap V(H_j)|=1$, and
 (conditions for being $K$-fat)   
    $d_G(V(H_i),V(H_j))\ge K$,  $d_G(P_{i,j},V(H_k))\ge K$ ($k\in \{1,2,3\}, k\neq i, j$) and the distance between any two paths $P_{1,2}$, $P_{2,3}$, $P_{1,3}$ is at least $K$.
Similarly, a graph $G$ has a {\em $K$-fat $K_{1,3}$-minor} ($K>0$) if there are four connected subgraphs $H_0$, $H_1$, $H_2$, $H_3$  and three simple paths $P_{1,0}$, $P_{2,0}$, $P_{3,0}$ in $G$ such that for each $i\in \{1,2,3\}$, $P_{i,0}$ has one end in $H_i$ and the other end in $H_0$ and  $|P_{i,0}\cap V(H_i)|=|P_{i,0}\cap V(H_0)|=1$, and 
 (conditions for being $K$-fat)   
    $d_G(V(H_i),V(H_j))\ge K$ for  $i,j\in \{0,1,2,3\}$ ($i\neq j$),  $d_G(P_{i,0},V(H_k))\ge K$ ($k\in \{1,2,3\}, k\neq i$) and the distance between any two paths $P_{1,0}$, $P_{2,0}$, $P_{3,0}$ is at least $K$. Clearly, in the definition of $K$-fat $K_{1,3}$-minor,  $H_1$, $H_2$, $H_3$ can be picked as singletons (each of them being just a vertex of $G$). 

Denote by $\mfi(G)$ the largest $K>0$ such that $G$ has a $K$-fat $K_3$-minor or a $K$-fat $K_{1,3}$-minor. Call it the {\em $(K_3,K_{1,3})$-minor fatness index} of $G$.
The following two lemmas provide bounds on $\mfi(G)$ using $\mci(G)$. 

\begin{lemma} \label{lm:mci_fat}
	Let $G$ be a graph with $\mci(G)=r$. Then, $G$ has neither $K$-fat $K_3$-minor nor $K$-fat $K_{1,3}$-minor for $K>r$.  Consequently, $\mfi(G)\le \mci(G)$ for every graph $G$.
\end{lemma}
\begin{proof} Assume $G$ has a $K$-fat $K_3$-minor ($K>r$) formed by three connected subgraphs $H_1$, $H_2$, $H_3$ and three simple paths $P_{1,2}$, $P_{2,3}$, $P_{1,3}$ (see the definition above). Consider a middle vertex $x$ of $P_{1,2}$, a middle vertex $y$ of $P_{2,3}$, and a middle vertex $z$ of $P_{1,3}$. Since $\mci(G)=r$, without loss of generality, we may assume that disk $D_G(x,r)$ intersects every path of $G$ connecting $y$ with $z$. Hence, there must exist a vertex $w$ in $P_{2,3}\cup P_{1,3}\cup V(H_3)$ with $d_G(w,x)\le r$. Since $r<K$, a contradiction with $K$-fatness of this $K_3$-minor arises.  

Assume $G$ has a $K$-fat $K_{1,3}$-minor ($K>r$) formed by four connected subgraphs $H_0$, $H_1$, $H_2$, $H_3$ and three simple paths $P_{1,0}$, $P_{2,0}$, $P_{3,0}$ (see the definition above). Consider a vertex $x$ in $H_1$, a vertex $y$ of $H_{2}$, and a vertex $z$ of $H_3$. Since $\mci(G)=r$, without loss of generality, we may assume that disk $D_G(x,r)$ intersects every path of $G$ connecting $y$ with $z$. Hence, there must exist a vertex $w$ in $V(H_2)\cup P_{2,0}\cup V(H_0)\cup P_{1,0}\cup V(H_3)$ with $d_G(w,x)\le r$. Since $r<K$, a contradiction with $K$-fatness of this $K_{1,3}$-minor arises.
\qed 
\end{proof} 


\begin{lemma} \label{lm:pat_fat}
	Let $G$ be a graph  with $\mci(G)>4K-1$. Then, $G$ has a $K$-fat $K_3$-minor or a $K$-fat $K_{1,3}$-minor.  
\end{lemma}
\begin{proof}
Since $\mci(G)> 4K-1$, there must exist in $G$ three vertices $u,v,w$ and three paths $P(u,v)$, $P(v,w)$, $P(u,w)$ connecting corresponding vertices such that 
$P(u,w)$ avoids disk $D_G(v,4K-1)$, $P(v,w)$ avoids disk $D_G(u,4K-1)$ and $P(u,v)$ avoids disk $D_G(w,4K-1)$. 
In particular, $d_G(u,v)\ge 4K$, $d_G(u,w)\ge 4K$, $d_G(w,v)\ge 4K$ and the length of each path $P(u,w)$, $P(u,v)$, $P(v,w)$ is at least $4K$. 
%
%
Let $v_u\in P(u,v)$ and $v_w\in P(w,v)$ be vertices of $P(u,v)$ and $P(w,v)$ such that  $d_G(v,v_u)=d_G(v,v_w)=\lfloor\frac{3}{2}K\rfloor$ and $d_{P(u,v)}(v,v_u)$ and $d_{P(w,v)}(v,v_w)$ are maximized (i.e., they are vertices of $D_G(v,\lfloor\frac{3}{2}K\rfloor)$ furthest from $v$ along those paths).  
Define similarly vertices $u_v\in P(u,v)$, $u_w\in P(u,w)$ and vertices $w_v\in P(w,v)$, $w_u\in P(u,w)$. See Figure \ref{fig:three}(a) for an illustration.   Denote by $P(v_u,u_v)$ a subpath of $P(u,v)$ between  $v_u$ and $u_v$. Similarly define paths $P(v_w,w_v)$ and  $P(u_w,w_u)$. 
Since $d_G(u,v)\ge 4K$, $d_G(u,w)\ge 4K$ and  $d_G(w,v)\ge 4K$, we get $d_G(v_u,u_v)\ge K$, $d_G(v_w,w_v)\ge K$ and  $d_G(u_w,w_u)\ge K$ and, hence,  the length of each path $P(v_u,u_v)$, $P(v_w,w_v)$ and  $P(u_w,w_u)$ is at least $K$. 

First assume that two of the three paths $P(v_u,u_v)$, $P(v_w,w_v)$,   $P(u_w,w_u)$  are less than $K$ close to each other, say, $d_G(P(u_v,v_u),P(v_w,w_v))<K$. Consider two vertices $x\in P(u_v,v_u)$ and $y \in P(v_w,w_v)$ such that $d_G(x,y)<K$ (see Figure \ref{fig:three}(b)). We claim, in this case,  that $G$ has a $K$-fat $K_{1,3}$-minor. Indeed, set  
$H_1:=\{v\}$, $H_2:=\{u\}$, $H_3:=\{w\}$, and let $H_0$ be a subgraph of $G$ formed by $P(u_v,v_u)$, $P(v_w,w_v)$ and any shortest path $Q(x,y)$ connecting $x$ and $y$ in $G$. These connected subgraphs $H_1$, $H_2$, $H_3$, $H_0$ of $G$ together with some shortest paths $Q(v,v_u)$, $Q(u,u_v)$, $Q(w,w_v)$, connecting corresponding vertices,  form a $K_{1,3}$-minor in $G$. Note that the length of each of those shortest paths $Q(v,v_u)$, $Q(u,u_v)$, and $Q(w,w_v)$ is $\lfloor\frac{3}{2}K\rfloor.$
We have $d_G(u,v)\ge 4K$, $d_G(u,w)\ge 4K$, $d_G(w,v)\ge 4K$. 
If $d_G(Q(v,v_u),Q(u,u_v))<K$, then $d_G(v,u)\le \lfloor\frac{3}{2}K\rfloor+d_G(Q(v,v_u),Q(u,u_v))+\lfloor\frac{3}{2}K\rfloor< 4K$, which is impossible. So, $d_G(Q(v,v_u),Q(u,u_v))\ge K$ must hold. Similarly, $d_G(Q(v,v_u),Q(w,w_v))\ge K$ and $d_G(Q(w,w_v),Q(u,u_v))\ge K$ must hold. We also have  $d_G(u,P(v,w))\ge 4K$, $d_G(w,P(u,v))\ge 4K$ and  $d_G(v,P(u,w))\ge 4K$. Hence, to see that this $K_{1,3}$-minor is $K$-fat, we need only to show that path $Q(x,y)$ is at distance at least $K$ from each of $u,v,w$.  We cannot have $d_G(w,Q(x,y))<K$ since then we get $d_G(x,w)<K+K=2K$, which is impossible as $x\in P(u,v)$. Similarly, we cannot have $d_G(u,Q(x,y))<K$. Assume, by way of contradiction, $d_G(v,Q(x,y))<K$. Then, there is a vertex $z\in Q(x,y)$ with $d_G(v,z)<K$. Assuming, without loss of generality,  $d_G(x,z)\le d_G(z,y)$, we get  $d_G(v,x)\le d_G(v,z) + d_G(z,x)\le K-1+(K-1)/2=3/2(K-1)<\lfloor\frac{3}{2}K\rfloor$. The latter contradicts the choice of vertex $v_u$ as being a vertex of $P(u,v)\cap D_G(v,\lfloor\frac{3}{2}K\rfloor)$  furthest (along the path $P(u,v)$) from $v$. This proves that the $K_{1,3}$-minor constructed is $K$-fat. 

Now, we can assume that all three paths $P(v_u,u_v)$, $P(v_w,w_v)$,   $P(u_w,w_u)$  are pairwise at distance at least $K$. In this case, we can build a $K$-fat $K_3$-minor in $G$. Set $H_v:=G[D_G(v,\lfloor\frac{3}{2}K\rfloor)]$, $H_u:=G[D_G(u,\lfloor\frac{3}{2}K\rfloor)]$, $H_w:=G[D_G(w,\lfloor\frac{3}{2}K\rfloor)]$. It is easy to see that these connected subgraphs $H_v,H_w,H_u$ and paths $P(v_u,u_v)$, $P(v_w,w_v)$,   $P(u_w,w_u)$ form a $K$-fat $K_3$-minor in $G$. 
Recall that $d_G(v,P(u,w))\ge 4K$,  $d_G(u,P(v,w))\ge 4K$  and $d_G(w,P(u,v))\ge 4K$. Hence, if 
$d_G(V(H_v),V(H_u))< K$, then $d_G(v,u)< \lfloor\frac{3}{2}K\rfloor+K+ \lfloor\frac{3}{2}K\rfloor\le 4K$, which is impossible. If  $d_G(V(H_v),P(u_w,w_u))< K$, then $d_G(v,P(u,w))< \lfloor\frac{3}{2}K\rfloor+K< 3K$, which is also impossible. By symmetries, the $K_3$-minor constructed is $K$-fat.  \qed
\end{proof} 

  \begin{figure}[htb]
    \begin{center} 
      \begin{minipage}[b]{15cm}
        \begin{center} 
          \vspace*{-10mm}
          \includegraphics[height=17cm]{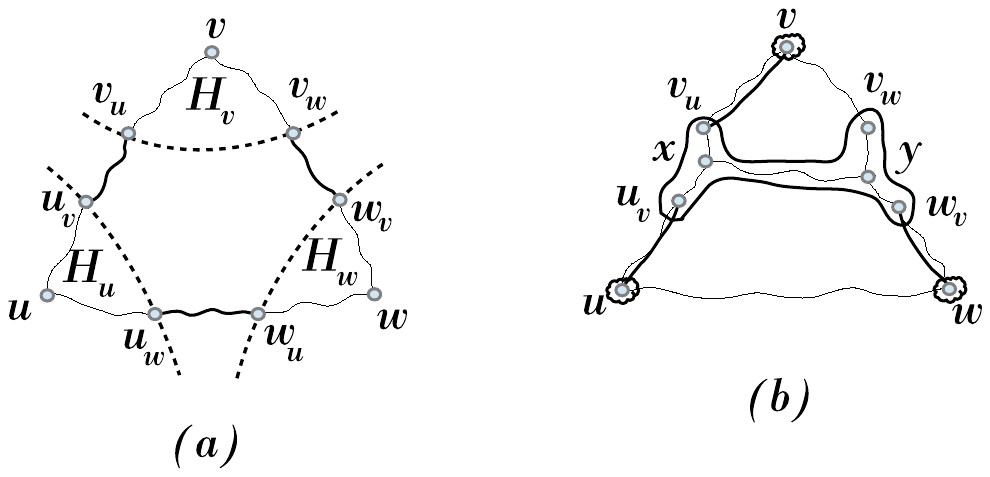}
        \end{center} \vspace*{-123mm}
        \caption{\label{fig:three} Illustrations to the proof of Lemma \ref{lm:pat_fat}. (a) a fat $K_3$-minor. (b) a fat $K_{1,3}$-minor.}  
      \end{minipage}
    \end{center}
   \vspace*{-5mm}
  \end{figure}
 
Note that, in the proof of Lemma \ref{lm:pat_fat}, we constructed very specific $K$-fat $(K_3,K_{1,3})$-minors. In our $K$-fat $K_3$-minor, the connected subgraphs  $H_1,H_2,H_3$ are disks. In our $K$-fat $K_{1,3}$-minor, the connected subgraphs  $H_1,H_2,H_3$ are singletons and the paths $P_{i,0}$, $i=1,2,3$, are shortest paths. Even more specific $K$-fat $(K_3,K_{1,3})$-minors were obtained in \cite{Diestel++}.  It was shown 
that if a graph $G$ contains no ($\ge K$)-subdivision of $K_3$ as a geodesic
subgraph and no ($\ge 3K$)-subdivision of 
$K_{1,3}$ as a 3-quasi-geodesic subgraph, then $\pb(G)\le 18K + 2$ (see \cite{Diestel++} for details and definitions). 

From Lemma \ref{lm:pat_fat}, we immediately get the following corollary. 

\begin{corollary} \label{cor: mci-fmi}
If $G$ has neither $K$-fat $K_3$-minor nor  $K$-fat $K_{1,3}$-minor, then  $\mci(G)\le 4K-1$. In particular, $\mci(G)\le 4\cdot\mfi(G)+3$. 
\end{corollary} 
\begin{proof} The first part of Corollary \ref{cor: mci-fmi} follows from Lemma \ref{lm:pat_fat}. For the second part, let $\mfi(G)=K-1$. Then, $G$ has neither $K$-fat $K_3$-minor nor  $K$-fat $K_{1,3}$-minor. By the first part, $\mci(G)\le 4K-1=4(K-1)+3=4\cdot\mfi(G)+3$. \qed
\end{proof} 

Combining  Lemma \ref{lm:mci_fat} and  Corollary \ref{cor: mci-fmi} with Theorem \ref{th:pat-pl-mci}, we get. 

\begin{theorem} \label{th:mfi-all}
	For every graph $G$,  
$$\mfi(G)\le \mci(G)\le 4\cdot\mfi(G)+3,$$    
  $$\mfi(G)\le \pl(G)\le 16\cdot\mfi(G)+10,$$  
  $$\frac{\mfi(G)}{2}\le\pat(G)\le 8\cdot\mfi(G)+5,$$  
$$\frac{\mfi(G)-1}{2}\le\adc(G)\le 16\cdot\mfi(G)+10,$$ 
$$\mfi(G)\le\Delta(G)\le 16\cdot\mfi(G)+12,$$ 
$$\frac{\mfi(G)}{4}\le\dsp(G)\le \dpr(G)\le 8\cdot\mfi(G)+5,$$ 
$$\frac{\mfi(G)}{2}\le\pcc(G)\le 16\cdot\mfi(G)+10.$$
\end{theorem}

\section{Concluding remarks and open questions} \label{sec:concl}
We saw that several graph parameters are coarsely equivalent to path-length. If one of the parameters from the list $\{\pl(G),$ $\pb(G),$ $\spb(G),$ $\ipl(G),$ $\ipb(G),$ $\Delta(G),$ $\rho(G),$  $\adc(G),$  $\dsp(G),$ $\dpr(G),$ $\pat(G),$ $\pcc(G),$ $\mci(G),$ $\mfi(G)\}$ is bounded for a graph $G$, then all other parameters are bounded. We saw that, in fact, all those parameters are within small constant factors from each other. 
Two questions are immediate.  
\begin{itemize}
\item[1.]   Can constants in those inequalities be further improved? In particular, is $\pcc(G)\le\pat(G)+1$  true, is $\pcc(G)\le2\cdot\mci(G)$  true? 
\item[2.] Are there any other interesting graph parameters that are coarsely equivalent to path-length? 
\end{itemize}

We know that for a given graph $G$, the parameters $\pl(G),$ $\pb(G),$ $\spb(G),$ $\ipl(G),$ $\ipb(G),$ $\dsp(G)$ are NP-hard to compute exactly, while the parameters $\Delta(G),$ $\rho(G),$ $\dpr(G),$ $\pat(G),$ $\pcc(G),$ $\mci(G)$ are computable in polynomial time and provide constant factor approximations for those hard to compute parameters. 
\begin{itemize}
\item[3.] What is the complexity of computing  $\adc(G),$ $\mfi(G)\}$? Can non-trivial inapproximability results be obtained for those hard to compute parameters?  
\end{itemize}

Following \cite{CTS1}, where a notion of $k$-tree-breadth was introduced, we can introduce a notion of $k$-path-breadth. The {\em $k$-breadth} of a path-decomposition $\cP(G)$  of a graph $G$ is the minimum integer $r$ such
that for each bag $B$ of $\cP(G)$, there is a set of at most $k$ vertices $C_B = \{v_B^1,\dots,v_B^k\}\subseteq V(G)$ such that for each $u\in B$, $d_G (u, C_B) \le r$ holds (i.e., each bag $B$ can be covered with at most $k$ disks of $G$ of radius at most $r$ each; $B\subseteq \cup_{i=1}^k D_G(v_B^i,r)$. The {\em $k$-path-breadth} of a graph $G$, denoted by $\pb_k(G)$, is the minimum of the $k$-breadth,
over all path-decompositions of $G$. Clearly, for every graph $G$, $\pb(G) = \pb_1(G)$ and $\pw(G) \le k - 1$ if and only if $\pb_k(G) = 0$ (each vertex in the bags of the path-decomposition can be considered as a center of a disk of radius 0). 
\begin{itemize}
\item[4.] It would be interesting to investigate if there exist  generalizations of  some graph parameters considered in this paper that coarsely describe the $k$-path-breadth.  
\end{itemize}

\medskip\noindent
{\sc{Note added:} }
When this paper was ready for submission to a journal, we learned about two recent interesting papers \cite{coarse-tw2,coarse-tree-width}. Authors of \cite{coarse-tree-width} prove stronger results: the $k$-path-breadth (the $k$-tree-breadth) of a graph $G$ is bounded if and only if there is an $(L,C)$-quasi-isometry (with $L$ and $C$ bounded) from $G$ to a graph $H$ with pseudo-path-width (with  pseudo-tree-width, respectively) at most $k-1$. See \cite{coarse-tree-width}    for details. The results of \cite{coarse-tw2} are similar but somehow weaker:  the $k$-path-breath (or $k$-tree-breadth) of a graph $G$ is bounded if and only if there is an $(L,C)$-quasi-isometry (with $L$ and $C$ bounded) from $G$ to a graph $H$ with path-width (or tree-width, respectively) at most $2k-1$.  See \cite{coarse-tw2} for details.


\newpage
{\bf Appendix A: Graph parameters considered} 

\begin{table} [htbp]
	\centering
	\begin{tabular}{| l | l |}
		\hline
Notation      & Name \\ 
\hline
\noalign{\smallskip}
 $\pl(G)$ & path-length of $G$ \\ 
\hline
\noalign{\smallskip}
 $\ipl(G)$ &  inner path-length of $G$ \\ 
\hline
\noalign{\smallskip}
 $\pb(G)$ & path-breadth of $G$ \\ 
\hline
\noalign{\smallskip}
$\ipb(G)$ &  inner path-breadth of $G$ \\ 
\hline
\noalign{\smallskip}
 $\spb(G)$ & strong path-breadth of $G$ \\  
\hline
\noalign{\smallskip}	
$\Delta_s(G)$ &  the length of path-decomposition $L^+(s, G)$ of $G$ with a start vertex $s$ \\
\hline
\noalign{\smallskip}	 
$\Delta(G)$ & minimum of $\Delta_s(G)$ over all vertices $s$ of $G$ \\
\hline 
\noalign{\smallskip}	
$\rho_s(G)$ & the breadth of path-decomposition $L^+(s, G)$ of $G$ with a start vertex $s$ \\
\hline  
\noalign{\smallskip}	 
$\rho(G)$ & minimum of $\rho_s(G)$ over all vertices $s$ of $G$ \\
\hline 
\noalign{\smallskip}	
$\adc(G)$ & additive distortion of embedding of $G$ to an unweighted caterpillar tree \\
\hline 
\noalign{\smallskip}	
$\pat(G)$ & power-AT index of $G$ \\
\hline 
\noalign{\smallskip}	
$\pcc(G)$ & power-cocomparability index of $G$ \\
\hline 
\noalign{\smallskip}	
$\dpr(G)$ & dominating-pair–radius of $G$ \\
\hline 
\noalign{\smallskip}	
$\dsp(G)$ & dominating-shortest-path–radius of $G$ \\
\hline 
\noalign{\smallskip}	
$\mci(G)$ & McCarty index of $G$ \\
\hline 
\noalign{\smallskip}	
$\mfi(G)$ & $(K_3,K_{1,3})$-minor fatness index of $G$ \\
\hline 
\end{tabular}
	\label{table:parameters}
\end{table}

{\bf Appendix B: Bounds that follow from the known before results} 
$$\pb(G) \leq \pl(G) \leq 2\cdot\pb(G)  \mbox{~~[trivial]} \mbox{~~ and~~ } \pb(G) \leq \spb(G) \leq 4\cdot\pb(G)\mbox{~\cite{Dulei2019}}$$ 
$$\pl(G) \leq \ipl(G)\le 2\cdot\pl(G)  \mbox{~\cite{BerSey2024}~~and~~} \pb(G) \leq \ipb(G)\le 2\cdot\pb(G)\mbox{~\cite{Diestel++}}$$ 
$$\pl(G)\le\Delta(G)\le 2\cdot\pl(G)\mbox{~\cite{DrKoLe2017}} \mbox{~~ and~~} \pb(G)\le\rho(G)\le 3\cdot\pb(G)\mbox{~\cite{DrKoLe2017}}$$
$$\dpr(G) \leq \pl(G)\mbox{~\cite{DrKoLe2017}} \mbox{~~ and~~ }\pat(G)\le 2\cdot\pl(G)-1\mbox{~\cite{DrKoLe2017}}$$ 
$$\pat(G)\le \pcc(G)\le 2\cdot\pat(G)\mbox{~\cite{Chang2003PowersOA,AT-free-first}}  $$
$$ \spb(G)\le 2\cdot\dsp(G)\mbox{~\cite{Dulei2019}} \mbox{~~ and~~} \mci(G)\le \pl(G)\mbox{~\cite{DrKoLe2017}}$$ 
\bigskip

\bigskip

\newpage 
{\bf Appendix C: Bounds 
from this paper}

(we underline parameters that are central in the inequalities) 
\bigskip

{\bf Bounds with $\adc(G)$} (Theorem \ref{th:adc-pl}) 
     $$\adc(G)\le \underline{\Delta(G)}\le    3\cdot \adc(G)+3$$
  $$\frac{\adc(G)}{2}\leq \frac{\Delta(G)}{2}\leq \underline{ \pl(G)}\leq 2\cdot\adc(G)+1$$  

{\bf Bounds with $\pat(G)$ (with $\pcc(G)$, $\dpr(G),\dsp(G)$)} (Theorem \ref{th:pat-pl} and  Corollary \ref{cor:ineq-pl-pat-dsp}) 
    $$\underline{\rho(G)}\le 2\cdot\dsp(G)+1\mbox{~and~} \dsp(G)\le  \underline{\dpr(G)}\le 3\cdot\dsp(G)$$ 
$$\dsp(G)\leq \dpr(G)\le \underline{\pl(G)}\le 4\cdot\dsp(G)\le 4\cdot\dpr(G)$$ 
$$\dsp(G)\leq \dpr(G)\le \underline{\Delta(G)}\le 4\cdot\dsp(G)+1\le 4\cdot\dpr(G)+1$$   
  $$\frac{\adc(G)-2}{2}\leq \underline{\dsp(G)}\leq \underline{\dpr(G)}\le  2\cdot\adc(G)+1$$ 
$$\frac{\adc(G)}{2}\leq \underline{\pat(G)}\le \underline{\pcc(G)}\le 3\cdot\adc(G)+2$$ 
$$\frac{\pl(G)}{2}\leq \underline{\pat(G)}\le 2\cdot\pl(G)-1$$
$$\frac{\Delta(G)-1}{2}\le \underline{\pat(G)}\le 2\cdot\Delta(G)-1$$
$$\pat(G)\le \underline{\pcc(G)} \le 2\cdot\pat(G)$$ 
$$\frac{\pl(G)}{2}\leq \pat(G)\le \underline{\pcc(G)}\le 2\cdot\pl(G)$$  
 $$\dsp(G)\le \dpr(G)\le \underline{\pat(G)}\le \underline{\pcc(G)}\le  6\cdot\dsp(G)\le  6\cdot\dpr(G)$$ 
$$ \frac{\Delta(G)-1}{2}\le  \underline{\pcc(G)}\le 2\cdot\Delta(G)$$ 

{\bf Bounds with $\mci(G)$ and $\mfi(G)$} (Theorem \ref{th:pat-pl-mci}, Corollary  \ref{cor:ineq-with-mci} and Theorem \ref{th:mfi-all}) 

(we assume $\mci(G)>0$, i.e., $G$ is not a path) 

$$\mfi(G)\le \underline{\mci(G)}\le 4\cdot\mfi(G)+3$$    
  $$\mfi(G)\le \mci(G)\le \underline{\pl(G)}\le 4\cdot\mci(G)-2\le 16\cdot\mfi(G)+10$$  
  $$\frac{\mfi(G)}{2}\leq\frac{\mci(G)}{2}\leq\underline{\pat(G)}\le 2\cdot\mci(G)-1\le 8\cdot\mfi(G)+5$$  
$$\frac{\mfi(G)-1}{2}\le\frac{\mci(G)-1}{2}\le \underline{\adc(G)}\le 4\cdot\mci(G)-2\le 16\cdot\mfi(G)+10$$ 
$$\mfi(G)\le \mci(G)\le\underline{\Delta(G)}\le 4\cdot\mci(G)\le 16\cdot\mfi(G)+12$$ 
$$\frac{\mfi(G)}{4}\le\frac{\mci(G)}{4}\le \underline{\dsp(G)}\le \underline{\dpr(G)}\le 2\cdot\mci(G)-1\le 8\cdot\mfi(G)+5$$ 
$$\frac{\mfi(G)}{2}\le\frac{\mci(G)}{2}\le \underline{\pcc(G)}\le  4\cdot\mci(G)-2\le 16\cdot\mfi(G)+10$$

\end{document}